\documentclass{article} % For LaTeX2e
\usepackage[left=1.25in, right = 1.25in, top = 1.25in, bottom = 1.25in]{geometry}

\usepackage{amsthm}
\usepackage{stmaryrd}
\usepackage{xcolor}

\usepackage{amsmath,amsfonts,bm}

\def\eqref#1{equation~\ref{#1}}
\def\1{\bm{1}}

\DeclareMathAlphabet{\mathsfit}{\encodingdefault}{\sfdefault}{m}{sl}
\SetMathAlphabet{\mathsfit}{bold}{\encodingdefault}{\sfdefault}{bx}{n}

\DeclareMathOperator*{\argmin}{arg\,min}

\usepackage{hyperref}
\usepackage{url}

\usepackage[all]{xy}
\usepackage{algorithm}
\usepackage{algpseudocode}
\usepackage{graphicx}
\usepackage{comment}

\title{Spherical Coordinates from Persistent Cohomology}

\author{Nikolas C. Schonsheck\thanks{University of Delaware, Department of Mathematical Sciences} \ and Stefan C. Schonsheck\thanks{University of California Davis, Department of Mathematics, TETRAPODS Institute of Data Science} }

\date{}

\newcommand{\ZZ}{\mathbb{Z}}
\newcommand{\FF}{\mathbb{F}}
\newcommand{\RR}{\mathbb{R}}
\newcommand{\CC}{\mathbb{C}}
\newcommand{\iso}{\cong}
\newcommand{\BC}{\text{\small{BC}}}
\newcommand{\sett}[1]{\left\{#1\right\}}
\newcommand{\abs}[1]{\left\vert#1\right\vert}
\newcommand{\norm}[1]{\left\Vert#1\right\Vert}
\newcommand{\bc}{\text{\sc bc}}

\theoremstyle{plain}
\newtheorem{thm}{Theorem}
\newtheorem{prop}[thm]{Proposition}

\theoremstyle{definition}
\newtheorem{defn}[thm]{Definition}
\newtheorem{remark}[thm]{Remark}

\DeclareMathOperator{\sk}{sk}
\DeclareMathOperator{\id}{id}

\begin{document}

\title{Spherical coordinates from persistent cohomology}

\author{Nikolas C. Schonsheck\thanks{University of Delaware, Department of Mathematical Sciences} \ and Stefan C. Schonsheck\thanks{University of California Davis, Department of Mathematics, TETRAPODS Institute of Data Science} }

\maketitle

\begin{abstract}
We describe a method to obtain spherical parameterizations of arbitrary data through the use of persistent cohomology and variational optimization. We begin by computing the second-degree persistent cohomology of the filtered Vietoris-Rips (VR) complex of a data set $X$ and extract a cocycle $\alpha$ from any significant feature. From this cocycle, we define an associated map $\alpha: VR(X) \to S^2$ and use this map as an infeasible initialization for a variational model, which we show has a unique solution (up to rigid motion). We then employ an alternating gradient descent/M\"{o}bius transformation update method to solve the problem and generate a more suitable, i.e., smoother, representative of the homotopy class of $\alpha$, preserving the relevant topological feature. Finally, we conduct numerical experiments on both synthetic and real-world data sets to show the efficacy of our proposed approach. 
\end{abstract}

\section{Introduction}

Representing high dimensional data in a lower dimension through nonlinear dimensionality reduction (NLDR) algorithms is a key step in understanding complex data. Different NLDR methods attempt to preserve different key properties of the data, such as global structure \cite{belkin2003laplacian}, local neighborhoods \cite{roweis2000nonlinear}, statistical properties \cite{van2008visualizing} or  manifold constraints \cite{mcinnes2018umap}, while simultaneously simplifying the representation of the information contained in the data set \cite{tsai2007dimensionality}. Recently, algorithms that can preserve various topological properties have been of particular interest \cite{Carlsson_topology_and_data, de2011persistent, falorsi2018explorations}. In many applications, preserving topological features of data, such as ``holes'' of various dimensions, is key to downstream tasks including classification, regression, and novel data generation \cite{schonsheck2019chart}. In this work, we generalize the ideas of de Silva et al. \cite{de2011persistent} to a higher dimension, producing topologically-invariant spherical parameterizations of data sets as represented by simplicial complexes.

Simplicial complexes are natural generalizations of graphs that can encode higher-dimensional information beyond pairwise relationships. From a theoretical standpoint, they are useful in that they provide a ``rigid'' model for topological spaces that can be represented by a finite sample of points from the space \cite{belkin2003laplacian}. Moreover, most reasonable spaces can be represented by a simplicial complex up to homotopy equivalence. In applications to data science, their advantages are even more numerous. For instance, whereas graphs contain only vertices and edges, which capture binary relations, simplicial complexes are built from vertices, edges, and their analogs in higher dimensions---namely $n$-simplices---and can encode relationships between any number of points. Simplicial complexes can also be easily represented on a computer, and there are tractable algorithms and software packages---such as Ripser \cite{ripser} and Eirene \cite{eirene}---that can calculate their topological invariants.

In this paper, we develop a method to find spherical parameterizations of simplicial complexes that have nontrivial topology as seen by second-degree cohomology. This method is based on an understanding that instructions for an oriented winding of triangular faces (i.e., 2-simplices) of a simplicial complex $X$ can be extracted from any cocycle in second-degree cohomology. Given a cocycle and list of faces, we define a variational problem with a unique optimum up to rigid motion. The resulting map $X \to S^2$ is useful for both data visualization and downstream analysis. In summary, our main contributions are as follows.

\begin{itemize}
    \item A generalization of the circular coordinates method of de Silva et al. \cite{de2011persistent} to spherical coordinates.
    \item A definition of \emph{discrete simplicial harmonic energy} for maps between simplicial complexes and the unit sphere based on oriented winding, or degree.
    \item A generalized \emph{discrete simplicial spring energy}, based on the energy above, which can be useful for noisy or non-uniformly sampled data including in the degree one, or circular, case.
    \item An efficient optimization method for minimizing these energies in both the degree one and two settings. 
    \item A proof that our smoothing procedure preserves the homotopical information encoded by the initial map.  
\end{itemize}

In Figure \ref{fig:comparison} we compare our algorithm to several popular nonlinear dimensionality reduction methods. We begin by randomly sampling 75 points on a sphere in $\RR^3$, then embed the sphere in $\RR^{50}$ and add random Gaussian noise in this higher dimension. We then use Laplacian Eigenmaps \cite{belkin2003laplacian}, ISOMAP \cite{balasubramanian2002isomap}, t-SNE \cite{van2008visualizing}, UMAP \cite{mcinnes2018umap} and our method to embed the noisy data back into $\RR^{3}$. Although some of these methods produce results that are roughly spherical, ours is the only one which maps exactly to $S^2 \in \RR^3$ and also is guaranteed to preserve the underlying topology. 

\begin{figure}
    \centering
    \includegraphics[width=.95\linewidth]{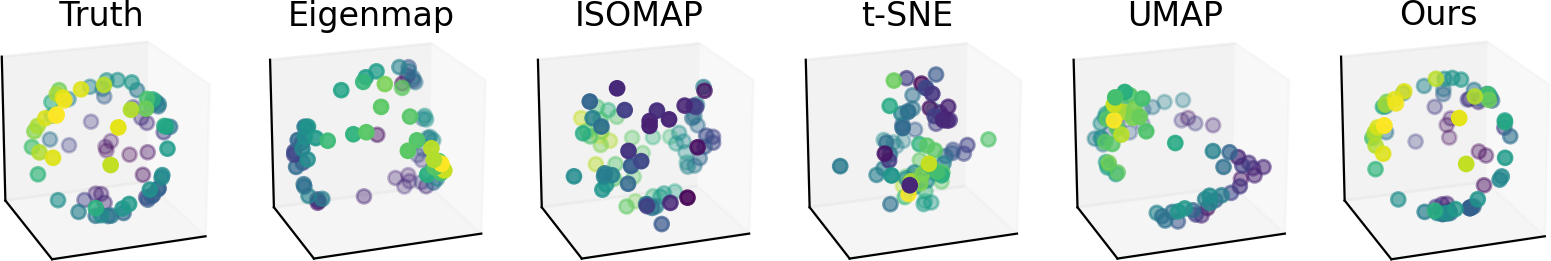}
    \caption{Nonlinear dimensionality reduction of a noisy sphere embedded in $\RR^{50}$ using several popular techniques and our methodology.}
    \label{fig:comparison}
\end{figure}

\subsection{Related Work}

The fundamental idea of this paper---that natural nonlinear coordinate systems for data can be extracted from persistent cohomology---goes back to \cite{de2011persistent}, where the authors recover circular coordinate systems from first-degree persistent cohomology. In more detail, de Silva et al. leverage the fact that the first integral Eilenberg-MacLane space $K(\ZZ,1)$ is homotopy equivalent to $S^1$ and the natural isomorphism $H^1(X; \ZZ) \ \iso \ [X, K(\ZZ,1)]$ given by Brown Representability to, first, define an initial map $\alpha \colon X \to S^1$ that represents a significant topological feature of the data set $X$. They then determine the ``smoothest'' such $\alpha$ by finding its harmonic representative in $H^1(X; \RR)$. Perea \cite{perea_circular} strengthens these results using the theory of principal $\ZZ$-bundles to design a more efficient and robust method for producing circular coordinates, known as ``sparse circular coordinates.''

The work above focuses on first-degree persistent cohomology. In this paper, we are interested in what can be obtained from second-degree persistent cohomology. The reason this is not a straightforward generalization is that in degree two, Brown Representability gives a natural isomorphism $H^2(X; \ZZ) \ \iso \ [X, K(\ZZ,2)]$ and whereas $K(\ZZ, 1) \simeq S^1$ is amenable to direct calculation, $K(\ZZ, 2)$ is homotopy equivalent to the infinite-dimensional projective space $\CC P^\infty$. Recent results of Perea \cite{perea_projective} in the context of equivariant dimensionality reduction show that one can reduce this problem to working in a finite-dimensional projective space via a PCA-type algorithm for projective spaces. However, the question raised in \cite{de2011persistent} of whether spherical coordinates, i.e., parameterizations $X \to S^2$, can be obtained from second-degree cohomology calculations has remained open thus far. Our results answer this question in the affirmative.  

In geometric data processing, defining maps from manifolds (or discrete representation thereof) to spheres has been a popular genre of approach for many problems. Most relevant to this work is that of \cite{gu2004genus} which defines a discrete version of harmonic energy that can be minimized to find a conformal map from an arbitrary genus-zero $2$-manifold (represented as a triangular mesh) to the unit sphere. Other shape processing algorithms have further developed these ideas for quasi-conformal maps \cite{lam2014landmark}, non-genus-0 maps \cite{zeng2010ricci} and others settings \cite{gu2003global, gu2008computational}. More recently, \cite{baden2018mobius} developed a shape matching pipeline that utilizes a M{\"o}bius transform of the unit sphere to find correspondences between genus-zero 3D shapes. As we will see in Section \ref{sec_smoothing}, the group of M{\"o}bius transformations will be essential to our smoothing procedure since they preserve important topological information, and can be easily approximated as in \cite{gu2004genus}.

In the geometric deep learning community, there has been a recent surge of interest in auto-regressive methods that are topologically informed. Works such as ~\cite{falorsi2018explorations, davidson2018hyperspherical, rey2019disentanglement, connor2020representing} introduce autoencoders with various non-Euclidean latent spaces that are useful for problems when the underlying data topology is known and relatively simple. This assumption is too strong for many real-world problems, where the topology is not known a priori.  In another line of research, \cite{schonsheck2019chart, schonsheck2022semi} and \cite{moor2020topological} provide methods for representing data sampled from unknown and arbitrary topological spaces. However, these methods can be sensitive to hyper-parameter tuning, and extracting topological invariants from the trained models is often nontrivial. 

Additionally, all of these auto-regressive methods rely on training neural networks with stochastic gradient descent to approximate the desired map(s). In general, the topologies are not guaranteed to be preserved---rather, the authors show that there exist networks that can preserve them and provide numerical experiments to show that the networks often do so. The method presented here, instead, finds the map directly by solving a variational optimization problem and is guaranteed to preserve the homotopy class of the map. Our method does not, however, lend itself directly to novel data generation, but rather to representation learning. 

\subsection{Main Pipeline}\label{subsec_main_pipeline}
To help orient the reader and outline the key points of the paper, we summarize here the main steps used to extract spherical parameterizations of data from second-degree persistent cohomology. We freely use notation that will be defined in later sections, but hope the following will be a useful reference. 

Let $X_\bullet = X_0 \to X_1 \to X_2 \to \cdots$ be a filtered simplicial complex with each $X_n$ finite. The main interest of this paper is the case when $X_\bullet$ is the filtered Vietoris-Rips complex on a data set of interest.
\begin{enumerate}
\item First, compute the degree two barcode $\BC^2(X_\bullet; \FF_p)$ of $X_\bullet$ with coefficients in a finite field $\FF_p$ and identify a bar of interest $\rho \in \BC^2(X_\bullet; \FF_p)$. As in \cite{de2011persistent}, we recommend choosing $p > 2$ so that $1$ and $-1$ are distinguished. 
\item Fix an interval decomposition of $PH^2(X_\bullet; \FF_p)$ and choose a parameter $\epsilon$ during the lifetime of $\rho$ to obtain a cohomology class $[\alpha_p] \in H^2(X_\epsilon; \FF_p)$ corresponding to $\rho$ under the chosen interval decomposition.
\item Letting $X_\epsilon = X$ for notation's sake, lift $[\alpha_p]$ to an integral cohomology class $[\alpha] \in H^2(X; \ZZ)$ as detailed in Section \ref{subsec_lifting_to_integer}.
\item As in Definition \ref{defn_canonical_lift} and Proposition \ref{prop_lifting_to_three_skeleton}, use $[\alpha]$ to define an \emph{initial map} $\tilde{\alpha} \colon \sk_3(X) \to S^2$.
\item Let $\alpha_0 \colon \sk_2(X) \to S^2$ be the restriction of $\tilde{\alpha}$ to the 2-skeleton of $X$.
\item Use the methodology of Section \ref{sec_smoothing} to produce a sequence of maps $\alpha_i \colon \sk_2(X) \to S^2$ for $0 \leq i \leq N$ converging to a map $\omega:=\alpha_N$ of minimal harmonic (or spring) energy.
\item Note that by Theorem \ref{thm_minimization_preserves_homotopy_on_2_skeleton} we have homotopies $\alpha_i \simeq \alpha_{i+1}$ for $0 \leq i \leq N-1$.
\item Appeal to Theorem \ref{thm_extension_to_3_skeleton} to see that 1) for each $0 \leq i \leq N$, $\alpha_i$ can be extended to a map $\tilde{\alpha_i} \colon \sk_3(X) \to S^2$, where we take $\tilde{\alpha_0} = \tilde{\alpha}$, and that 2) there is a sequence of homotopies
\begin{align}
\tilde{\alpha} = \tilde{\alpha_0} \simeq \tilde{\alpha_1} \simeq \cdots \simeq \tilde{\alpha_N} = \tilde{\omega}
\end{align}
of maps $\sk_3(X) \to S^2$.
\item Obtain the spherical parameterization $\tilde{\omega} \colon \sk_3(X) \to S^2$ with $\tilde{\omega} \simeq \tilde{\alpha}$ as maps $\sk_3(X) \to S^2$.
\end{enumerate}

\subsection{Acknowledgments}
The authors would like to thank Chad Giusti for many helpful discussions, and Jose Perea for an enlightening conversation about the paper. We are also grateful to several anonymous referees, whose detailed and thoughtful feedback led to significant improvements in the clarity and cohesiveness of the paper. NC Schonsheck's work is supported by the Air Force Office of Scientific Research under award number FA9550-21-1-0266 and SC Schonsheck's work is supported by the US National Science Foundation grant DMS-1912747.

\section{Persistent Cohomology}
In this section, we briefly recapitulate the basics of persistent cohomology. For a more thorough introduction, see \cite{Carlsson_topology_and_data} and \cite{Zomorodian_Carlsson}. Another useful treatment is given in \cite{analogous_bars}.

\subsection{Background on Persistent Cohomology}

Let $X$ be a finite simplicial complex and $F$ a finite field. For each $n \geq 0$, denote by $X^n$ the set of $n$-simplices of $X$ and let $C^n(X; F)$, or $C^n$ for short, denote the $F$-vector space of set maps $X^n \to F$. That is,
\begin{align}
C^n(X; F) = \sett{\text{maps of sets } f \colon X^n \to F}
\end{align}
where vector addition and scalar multiplication are inherited from $F$. Further define coboundary maps $d^n \colon C^n \to C^{n+1}$ by
\begin{align}
\begin{gathered}
(d^0f)(ab) = f(b) - f(a)\\
(d^1f)(abc) = f(bc) - f(ac) + f(ab)\\
(d^2f)(abcd) = f(bcd) - f(acd) + f(abd) - f(abc)\\
\vdots
\end{gathered}
\end{align}
It is a straightforward check that each $d^n$ is $F$-linear, allowing us to define the following.

\begin{defn}
The \emph{degree-$n$ cohomology} of the simplicial complex $X$ with coefficients in $F$ is the $F$-vector space given by
\begin{align}
H^n(X; F) = H^n(X) = \ker(d^n)/\text{im}(d^{n-1})
\end{align}
\end{defn}

In the persistent setting, rather than working with a fixed simplicial complex $X$, we often consider a filtered simplicial complex $X_\bullet$, i.e., a sequence of inclusions
\begin{align}\label{eqn_filtered_complex}
X_\bullet = X_0 \to X_1 \to X_2 \to \cdots
\end{align}
where each $X_n$ is a simplicial complex.

\begin{remark}
Here there is a potential for confusion in notation. We emphasize that if $X$ is a simplicial complex, then $X^n$ generally denotes its set of $n$-simplices. On the other hand, if $X_\bullet$ is a filtered simplicial complex, then $X_n$ refers to a simplicial complex. Similarly, given a simplicial complex $X$, we use $\sk_n(X)$ to denote its $n$-skeleton, i.e., the set of $k$-simplices of $X$ appropriately glued together for $k \leq n$.
\end{remark}

Because cohomology is functorial, any diagram as in Equation (\ref{eqn_filtered_complex}) induces the following corresponding diagram in cohomology

\begin{align}\label{eqn_filtered_homology}
PH^n(X_\bullet; F) = H^n(X_0; F) \leftarrow H^n(X_1; F) \leftarrow H^n(X_2; F) \leftarrow \cdots
\end{align}
called the persistent cohomology of $X_\bullet$. It follows from Equation (\ref{eqn_filtered_homology}) that we can view $PH^n(X_\bullet; F)$ as a graded $F$-vector space
\begin{align}
PH^n(X_\bullet; F) = \bigoplus_{i=0}^\infty H^n(X_i; F)
\end{align}
together with structure maps $H^n(X_i; F) \to H^n(X_{i-1}; F)$ for each $i > 0$. In other words, $PH^n(X_\bullet; F)$ is an example of a persistence module. 
\begin{defn}\label{defn_persistence_module}
A \emph{persistence module} is a $\ZZ$-graded $F$-vector space $V^\bullet = \oplus_{i=-\infty}^{\infty}V^i$ with linear structure maps $\sett{f^i\colon V^i \to V^{i+1}}$. We say that such a module has \emph{finite support} if $V^i = 0$ for all but finitely many $i$. Note that $V^\bullet$ has the structure of an $F[t]$-module where $t$ acts on $V^i$ via the structure map $f^i$. 
\end{defn}

\begin{remark}
To see that $PH^n(X_\bullet; F)$ is a persistence module, take $V^{-i} = H^n(X_i; F)$ as in Definition \ref{defn_persistence_module}.
\end{remark}

Provided $F$ is a field, $F[t]$ is a principal ideal domain (PID). Thus, if $PH^n(X_\bullet; F)$ has finite support and each $X_i$ is a finite simplicial complex, then the Structure Theorem for finitely generated modules over a PID implies that one can choose bases for each $H^n(X_i; F)$ that are compatible with the $F[t]$-module structure maps. We say such a basis element $[\sigma] \in H^n(X_k; F)$ is born at parameter $k$ if it is not in the image of $H^n(X_{k+1}; F) \to H^n(X_k; F)$ and that $[\tau]\in H^n(X_{\ell+1}; F)$ dies at parameter $\ell$ if it is mapped to $[0]$ by $H^n(X_{\ell + 1}; F) \to H^n(X_\ell; F)$. It further follows from the Structure Theorem that the collection of birth/death pairs of the basis elements is unique, i.e., does not depend on the particular choice of bases. This collection of pairs defines the barcode of $PH^n(X_\bullet; F)$, i.e., a triple

\begin{align}
(\bc^n(X_\bullet; F), \beta, \delta)
\end{align}
where $\bc^n(X_\bullet; F)$ is a finite set of bars, which we think of as corresponding to elements of some choice of compatible bases, and $\beta, \delta \colon \bc^n(X_\bullet; F) \to \ZZ$ record the birth and death times of these basis elements. For those familiar with persistence theory, it is worth noting that cohomological birth (resp., death) parameters correspond to homological death (resp., birth) parameters and that many software packages such as Ripser reindex persistence diagrams to align with the homological convention.

In general, there are many choices of compatible bases for a persistence module such as $PH^n(X_\bullet; F)$. Since we will be working with specific representatives of cohomology classes, it will be beneficial to fix a choice of basis. We call such a choice an \emph{interval decomposition} of $PH^n(X_\bullet; F)$. (See \cite[2.5]{chazal_structure_and_stability} for a useful reference.) In this paper, we generally assume that the practitioner (or relevant software) has chosen a fixed interval decomposition for any considered persistence module. 

\subsection{Lifting to Integer Coefficients}\label{subsec_lifting_to_integer}
In practice, one generally computes persistent cohomology with coefficients in a finite field $\FF_p$. However, our process for defining an $S^2$-valued map, below, requires an integral cohomology class. Given a simplicial complex $X$, our strategy for obtaining an element of $H^2(X; \ZZ)$ from an element of $H^2(X; \FF_p)$ is the same as that of \cite[2.4]{de2011persistent}. In more detail, consider the Bockstein long exact sequence
\begin{align}\label{eqn_bockstein}
\cdots \to H^2(X; \ZZ) \to H^2(X; \FF_p) \xrightarrow{\beta} H^3(X; \ZZ) \xrightarrow{\cdot p} H^3(X; \ZZ) \to \cdots
\end{align}
induced by the short exact sequence of coefficient rings $0 \to \ZZ \xrightarrow{\cdot p} \ZZ \to \FF_p \to 0$. Consider an element $[\alpha_p] \in H^2(X; \FF_p)$. Provided $\beta[\alpha_p] = [0]$ in $H^3(X; \ZZ)$, for instance if $H^3(X; \ZZ)$ has no nontrivial $p$-torsion elements, it follows from the exactness of the sequence that $[\alpha_p]$ is in the image of $H^2(X; \ZZ) \to H^2(X; \FF_p)$ and hence can be lifted to an element of $H^2(X; \ZZ)$. In practice, given $[\alpha_p] \in H^2(X; \FF_p)$, we replace the coefficients of $\alpha_p$ with their corresponding congruence classes in the range

\begin{align}
\sett{-(p-1)/2, \ldots, -1, 0, 1, \ldots, (p-1)/2}
\end{align}
and verify that the resultant cochain $\alpha \in C^2(X; \ZZ)$ satisfies the cocycle condition. Hence, $[\alpha]$ represents an element of $H^2(X; \ZZ)$. If this procedure fails, perhaps because $H^3(X; \ZZ)$ happens to have $p$-torsion, then we recommend simply choosing a different prime $p$ and redoing the calculation. However, we note that even in the absence of $p$-torsion one may still run into trouble---for instance, if $d^2\alpha = p\eta$ for some $\eta \in C^3(X; \ZZ)$. In this case, one can effect a ``repair'' as in \cite[2.4]{de2011persistent}. Thankfully, we have yet to encounter a real-world example where simply replacing coefficients as described above does not yield a cocycle.

\section{Obtaining a Map to the 2-Sphere}\label{sec:obtain}
Suppose $X_\bullet$ is a filtered finite simplicial complex, and select a bar $\rho \in \BC^2(X_\bullet; \FF_p).$ (While one can obtain a sphere-valued map from any bar by using the methodology below, we advise choosing a high-persistence bar, i.e., one that represents a significant topological feature. See Section \ref{subsec_feature_selection} for further discussion.) Fixing an interval decomposition of $P{H}^2(X_\bullet; \FF_p)$ and a parameter $\epsilon$ during the lifetime of $\rho$, we obtain a cohomology class $[\alpha_p] \in {H}^2(X_\epsilon; \FF_p)$ which we lift, as in Section \ref{subsec_lifting_to_integer}, to a class $[\alpha] \in H^2(X_\epsilon; \ZZ)$. Letting $X$ = $X_\epsilon$ for notation's sake, we now describe how to construct a sphere-valued map from $[\alpha]$.

The Brown Representability Theorem gives a natural isomorphism 
\begin{align}
H^2(X; \ZZ) \ \iso \ [X, K(\ZZ,2)]
\end{align}
where the latter term denotes the set (in this case, group) of homotopy classes of maps from $X$ into the Eilenberg-MacLane space $K(\ZZ,2) \simeq C\mathbb{P}^\infty$. Thus, the cohomology class $[\alpha]$ determines a unique homotopy class of maps $X \to K(\ZZ,2)$, and the goal now is to lift this to a sphere-valued map. Since there is a canonical inclusion $S^2 \to K(\ZZ,2)$, we can phrase this as searching for the dotted map below.
\begin{align}
\xymatrix{
& S^2 \ar[d]\\
X \ar@{.>}^-{\tilde{\alpha}}[ur] \ar^-\alpha[r] & K(\ZZ,2)
}
\end{align}
In general, there will be obstructions to lifting $\alpha$ on all of $X$ because $S^2$ has (infinitely) many nontrivial higher homotopy groups. However, since the inclusion $\sk_3(X) \to X$ induces an isomorphism $H^2(X; \ZZ) \xrightarrow{\iso} H^2(\sk_3(X); \ZZ)$, it is reasonable to instead define a lift from $\sk_3(X)$ which, in fact, always exists. To see this, first note that given a generic 2-simplex $\varphi$ and integer $n$, there is a homotopy-unique map $\varphi \to S^2$ defined by collapsing the boundary of $\varphi$ to obtain a 2-sphere, then mapping to $S^2$ via the degree $n$ map. Note that the degree $n$ map $S^2 \to S^2$ is the map that ``wraps $S^2$ around itself $n$ times,'' analogous to the map $S^1 \to S^1$ of winding number $n$. (More explicitly, this map can be obtained by first mapping $S^2$ to the wedge $\vee_{i=1}^{\abs{n}} S^2$ via the canonical ``pinch map'' and then mapping each constituent of the wedge to $S^2$ via the identity or antipodal map depending on the sign of $n$.) We now define the desired lift $\tilde{\alpha}$.

\begin{defn}\label{defn_canonical_lift}
Let $\alpha \in Z^2(X; \ZZ)$ be a 2-cocycle and fix a base-point $p \in S^2$. Let $\hat{\alpha}\colon\sk_2(X) \to S^2$ be given by mapping each vertex and 1-simplex of $X$ to $p$ and mapping each 2-simplex $\varphi$ of $X$ to $S^2$ as described above with $n=\alpha(\varphi)$. The following shows that this map can then be extended to a map $\tilde{\alpha} \colon\sk_3(X) \to S^2$. We call this extension $\tilde{\alpha}$ the \emph{canonical lift} of $\alpha$ and, in the context of Section \ref{subsec_main_pipeline}, the \emph{initial map}.
\end{defn}

\begin{prop}\label{prop_lifting_to_three_skeleton}
    The map $\hat{\alpha}$ defined on $\sk_2(X)$ in Definition \ref{defn_canonical_lift} can be extended to $\sk_3(X)$. 
\end{prop}
\begin{proof}

    Let $\sigma = abcd$ be a 3-simplex in $X$. The map $\hat{\alpha}$ can be extended to $\sigma$ if and only if the map from the boundary of $\sigma$ to $S^2$ is nullhomotopic. Since the boundary of $\sigma$ is homeomorphic to $S^2$, this map will be determined up to homotopy by its degree (by the Hopf degree theorem) and be nullhomotopic if and only if it is of degree 0. The fact that $\alpha$ is a cocycle guarantees exactly that. To see this, consider the composite below (Figure \ref{fig_map_to_sphere}), determined by $\hat{\alpha}$ on $\sk_2(X)$.

    \begin{figure}[htbp]
    \centering
    \includegraphics[scale=.5]{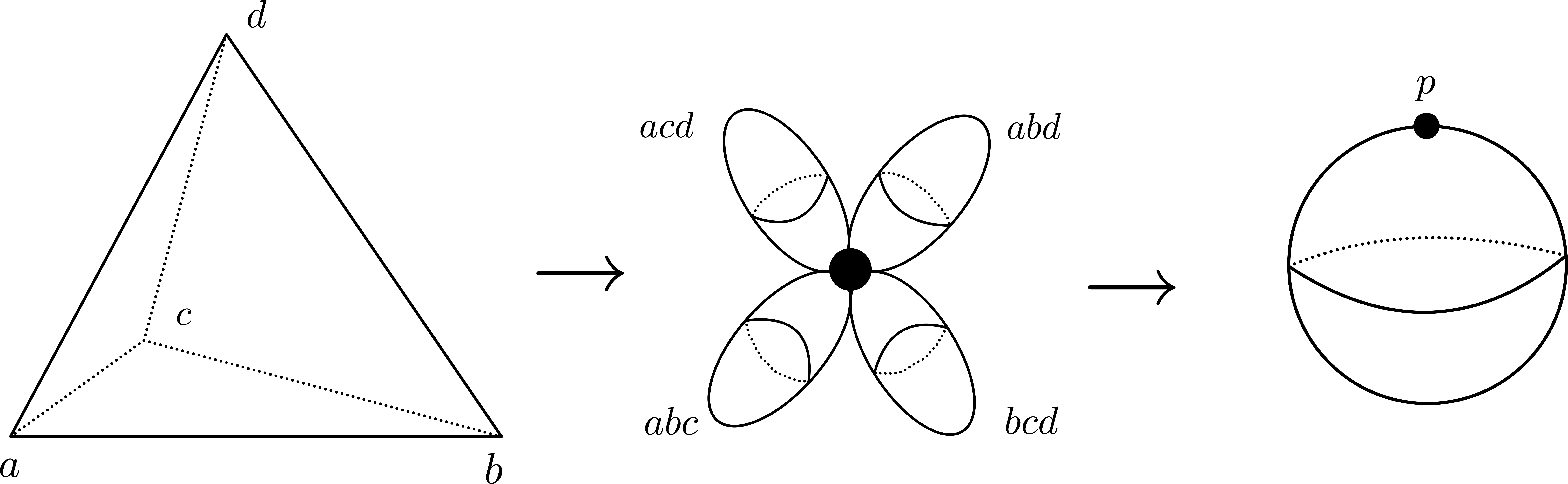}
    \caption{Given a 2-cocycle $\alpha$, we consider the corresponding map on the boundary of a generic 3-simplex $\sigma$.}
    \label{fig_map_to_sphere}
    \label{fig:map2s3}
    \end{figure}

    Let $f$ denote the first map in Figure \ref{fig_map_to_sphere} (which collapses the boundary of each bounding 2-simplex to a point to produce a wedge of spheres), $g$ the second (which maps each resultant sphere to $S^2$ via the degree map determined by $\hat{\alpha}$), and fix an orientation (i.e., an inward or outward facing normal) of $S^2$. We have induced maps on second-degree homology given by
    \begin{align}
    \xymatrix{
    \ZZ \ar^-{f_\ast}[r] & \ZZ_{bcd} \oplus \ZZ_{acd} \oplus \ZZ_{abd} \oplus \ZZ_{abc} \ar^-{g_\ast}[r] & \ZZ     
    }
    \end{align}

    The map $f_\ast$ is given by either $f_\ast(1) = (1, -1, 1, -1)$ or $f_\ast(1) = (-1, 1, -1, 1)$, depending on the chosen orientation.  In either case, $g_\ast$ is determined by
    \begin{align}\label{eqn_composite_map_on_homology}
    g_\ast(1, 1, 1, 1) = \alpha(bcd) + \alpha(acd) + \alpha(abd) + \alpha(abc)
    \end{align}
    and the cocycle condition guarantees that $g_\ast f_\ast = 0$. Indeed, we have
    \begin{align}
    g_\ast f_\ast(1) = \pm(\alpha(bcd) - \alpha(acd) + \alpha(abd) - \alpha(abc)) = 0
    \end{align}
    Repeating this argument for all 3-simplices of $X$ completes the proof.
\end{proof}

While Proposition \ref{prop_lifting_to_three_skeleton} shows that a lift $\tilde{\alpha}$ exists, it does not provide a description of the lift beyond the fact that it agrees with $\hat{\alpha}$ on $\sk_2(X)$. For a more explicit description, note that a null-homotopy on the boundary $\partial\sigma$ of a 3-simplex $\sigma$ as above is equivalent to a continuous map $h \colon (\partial\sigma \times [0,1]) / (\partial\sigma \times \sett{1}) \to S^2$. Since $(\partial\sigma \times [0,1]) / (\partial\sigma \times \sett{1})$ is homeomorphic to $\sigma$, the lift on a generic $3$-simplex $\sigma$ can be defined by composing such a homeomorphism with the map $h$.

In the context of Section \ref{subsec_main_pipeline}, we have now produced a sphere-valued map from the Vietoris-Rips complex on a data set of interest (more precisely, its 3-skeleton) to $S^2$. However, each vertex (i.e., data point) is mapped to the same selected basepoint in $S^2$. Our map is therefore of little use in data and downstream analysis. We now describe how to smooth this initial map to produce a useful spherical parameterization of the data.

\section{Smoothing the Initial Map}\label{sec_smoothing}
In this section, we detail Step (6) of the Main Pipeline (Section \ref{subsec_main_pipeline}). We assume that Steps (1)-(5) have already been performed, taking as granted an initial map $\tilde{\alpha}\colon\sk_3(X) \to S^2$, and letting $\alpha_0\colon\sk_2(X) \to S^2$ be the restriction of $\tilde{\alpha}$ to $\sk_2(X)$.

Although $\alpha_0$ maps our data (i.e., the vertices of $\sk_2(X)$) to the unit sphere, it is not very useful for visualization or analysis since all of the vertices are mapped to the same point. To develop a more useful mapping, we need to spread out the vertices across the sphere. We do so via a smoothing process that produces a sequence of maps $\alpha_i \colon \sk_2(X) \to S^2$ for $1 \leq i \leq N$ where $\alpha_N = : \omega$ is more amenable to analysis and downstream tasks. Crucially, we also show in Section \ref{sec_preserving_homotopical_info} that this smoothing process does not alter the homotopy class of the initial map.

\subsection{Harmonic Energy}

In order to obtain a useful representation of the data, we define the following notion of harmonic energy of a map from the 2-skeleton of a simplicial complex to the unit sphere. We are ultimately interested in minimizing this energy in order to smooth the map $\alpha_0$ and obtain a unique (up to rigid motion) final map $\alpha_N = \omega$.

\begin{defn}\label{defn_harmonic_energy}
Let $X$ be a finite simplicial complex. We define the \emph{discrete simplicial harmonic energy} of a map $f \colon \sk_2(X) \to S^2$ where $S^2$ is endowed with the usual geodesic metric as follows:
$$ \mathcal{E}_H (f) = \sum_{x \in \sk_2(X)^2} \frac{1}{2}||A(f (x))||^2  $$
where $A(f (x))$ is the area of the spherical triangle to which the 2-simplex $x$ is mapped.
\end{defn}

We assign an area of $4\pi$ (i.e., the whole sphere) to 2-simplices that are mapped via the degree 1 or $-1$ map as in 
Definition \ref{defn_canonical_lift}, and assign an area of 0 to 2-simplices that are mapped via the degree 0 map. Note that a map of degree 0 that collapses all of the vertices to a single point has zero area. Additionally, a triangle with two coincidental vertices has zero area, and therefore zero harmonic energy as well. This energy is a direct generalization of the discrete harmonic energy proposed in \cite{gu2004genus} in that if $\sk_2(X)$ is in fact a regular, air-tight triangle mesh, then the energies are identical. 

A spherical triangle with an area less than $4\pi$ can be uniquely identified by three points and its barycenter. (Note that without the barycenter, three points on the sphere define two spherically-complementary triangles). Given this information we can measure the sides lengths, $a,b,c$, and semiperimeter, $s=\frac{1}{2}(a+b+c)$, of the triangle using spherical trigonometry \cite{weisstein2008spherical}. Finally, we employ L'Huilier's Theorem \cite{lhuilier1812maemoire} to compute the spherical excess of the triangle. That is:
$
\mathrm{E(x)} = 4 \arctan{ \big( \tan{\frac{s}{2}} \tan{\frac{s-a}{2}} \tan{\frac{s-b}{2}} \tan{\frac{s-c}{2}} \big)^{\frac{1}{2}} }.
$
Since we are working on a unit sphere the spherical excess is equal to the surface area of the spherical triangle. 

\subsection{Spring Energy}

In some cases, it is useful to replace harmonic energy with a notion of spring energy so each simplex may have some nonzero rest-state area and so that we can tune a ``pull'' parameter (i.e., spring constant). This is especially useful in cases where the data is sampled from ``oblong'' shapes (see examples in Section \ref{sec:numerical}) and prevents data points from clumping too closely together. Although this introduces additional parameters into the model, we have found that parameter selection is not particularly onerous and allows us to to apply this scheme to a wider variety of domains. 

\begin{defn}\label{defn_spring_energy}
Let $X$ be a finite simplicial complex. We define the \emph{discrete simplicial spring energy} of a map $f \colon \sk_2(X) \to S^2$ as: 
$$ \mathcal{E}_S (f) = \sum_{x \in \sk_2(X)^2} \frac{1}{2}||k(A(f(x)) -R)||^2   $$
where $k$ is the spring constant, and $R$ is the rest area of the spring.
\end{defn}
We will show later in Section \ref{sec:numerical} that by tuning $k$ and $R$ we can develop better representations of data with nonuniform sampling densities or noise in the barcode. Additionally, using $k > 1$ can speed up the convergence of the optimization routine that we detail in the next section. Note that we can recover the harmonic energy formulation from the spring energy by setting  $R=0$ and $k=1$.

\subsection{Optimization Routine}\label{sec:opt}

Given $\alpha_0\colon\sk_2(X) \to S^2$ as in Section \ref{subsec_main_pipeline}, minimizing harmonic (or spring) energy alone is not enough to uniquely determine a new, smoother map to $S^2$ since the solution is not unique. The minimizers do, however, form a M{\"o}bius group \cite{gu2004genus}. Later, we will show that M{\"o}bius transformations preserve relevant topological information, but for now we use it to define a unique representation within this group. First we stipulate that the final map $\alpha_N = \omega$ must have zero center of mass, that is, we add the constraint $\sum_{x\in \sk_2(X)} \omega(x) = [0,0,0]$ when we view $S^2$ as the unit sphere in $\RR^3$. This solution is unique up to rigid motion, which we can fix later during post-processing (see Section \ref{sec:numerical}). Thus, given the map $\alpha_0  \colon \sk_2(X) \to S^2$, the optimization problem is:

\begin{equation}
    \argmin_{\alpha'\colon \sk_2(X)   \to S^2} \mathcal{E}\big(\alpha'\big) 
\end{equation}
subject to the constraints that 1) $\alpha'$ has zero center of mass and 2) $\alpha'$ can be extended to a map on $\sk_3(X)$, and that this extension is homotopic to the initial map $\tilde{\alpha}\colon\sk_3(X) \to S^2$. \ 

To solve this problem we use an alternating minimizaiton scheme beginning with the map $\alpha_0$ (see Section \ref{subsec_main_pipeline}). Note this is an infeasible starting point since it is not zero-centered. Moreover, since all of the vertices are mapped to the same point on the sphere, it is not possible to zero-center the vertices without moving them off of the unit sphere. Instead, we first decrease the energy by performing a few steps of manifold gradient descent \cite{bonnabel2013stochastic} without this constraint. This allows the vertices to spread which makes the enfrocemnt of this constraint tractable. To approach the feasible region of the zero center of mass constraint, we employ a M{\"o}bius transform that approximates the true transformation which centers the mass as in \cite{gu2004genus}. By choosing the step-size carefully we are able to ensure convergence to a unique minimizer. 

Consider a single spherical triangle with a nonzero area. The negative gradient of the harmonic energy pulls each vertex in the direction of the barycenter. Moreover, the magnitude of the gradient is proportional to the area of the triangle. Thus, to perform a step of the gradient update, while remaining on the sphere, we need to find the geodesic between each vertex and the barycenter as well as the area of the spherical triangle. In practice, we compute this direction as the spherical tangent to the geodesic, take a small step in the direction of the gradient in the tangent plane, and then project back to the sphere. So for vertex $p$, triangle $x$, and barycenter $b$, we compute $g = p-b - \langle p-b, p \rangle$, and then the gradient terms:

\begin{equation}
    \nabla_{p^{\perp},x} \mathcal{E}_H = \mathrm{A}(x) \frac{g}{||g||_2} \quad \text{and} \quad \nabla_{p^{\perp},x} \mathcal{E}_S =  k(\mathrm{A}(x)-D) \frac{g}{||g||_2}
\end{equation}
where $A(x)$ is the area of the triangle. Next, to compute the gradient update for the entire complex, we simply sum these gradients pointwise over all of the individual triangles $x$. This process is embarrassingly parallelizable and can be computed extremely efficiently. 

For a triangle with nonzero winding (i.e., a triangle which has wrapped around the entire sphere), we have the difficulty that there is not a unique geodesic between the vertex and the barycenter (since the vertices and barycenter are antipodal). However, since we have previously defined a positive orientation and negative orientation for these elements, we can choose a canonical direction of the gradient for one orientation and use its mirror image for the opposite orientation (see Figure \ref{fig:directions}) in order to define a gradient direction for each orientation. After one step in the  gradient direction, the triangle will no longer cover the entire sphere and can be treated as those above. Note that a different choice of orientation directions will lead to an equivalent final map  up to some reflection and/or rotation. 

Finally, to perform a step in the gradient direction, we move each vertex in its tangent plane in the direction determined by the gradient, scaled by a step-size parameter, $\delta_g$, and then normalize each point so that they once again lie on the unit sphere. After each gradient step, we update the barycenters of each spherical triangle by finding the Euclidean barycenter, normalizing it to the sphere, then choosing the correct pole based on the previous barycenter.

Computing the exact M{\"o}bius transform which centers the mass of a map $\sk_2(X) \to S^2$ is nonlinear and computationally expensive. Instead, following the method of \cite{gu2004genus}, we approximate the exact M{\"o}bius transform with one that is more easily computable and improve this approximation at each iteration. First, we compute the current center of mass $c$ and then move all of the points toward the center of mass by a step size of $\delta_m$. Finally, we project each point back onto the unit sphere using the Gauss map. In order to preserve the triangle orientation, the barycenters are updated similarly. Since the inital map, $\alpha_0,$ begins with all of the points being mapped to the same basepoint, we compute a few iterations of the gradient descent before beginning the alternating process. The full algorithm is presented in Algorithm \ref{apdx:alg}

\begin{figure}[ht]
    \centering
    \includegraphics[scale=0.55]{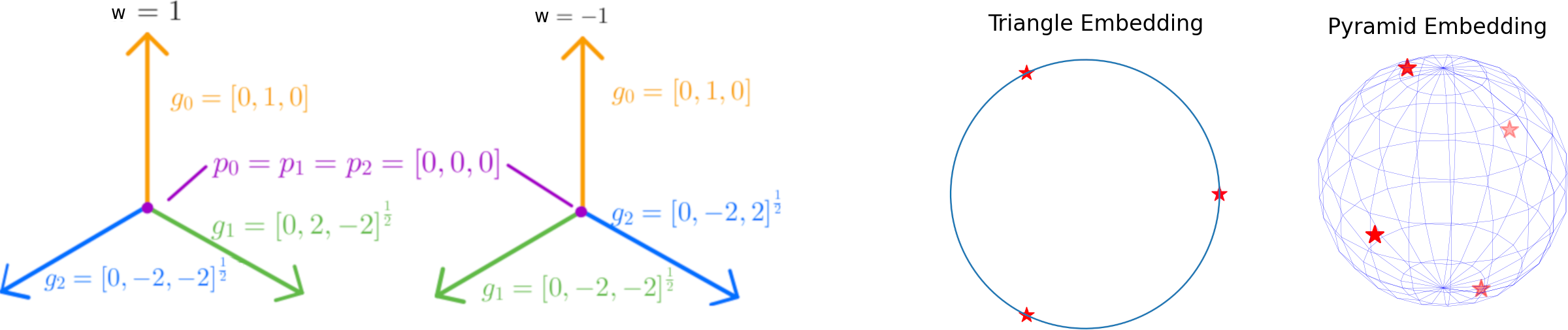}
    \caption{\textbf{Left}: Our choice of canonical directions for positive and negative orientations for a triangle with nonzero winding ($w = \pm 1)$, vertices at $[1,0,0]$ and barycenter $[-1,0,0]$. The outward normal is chosen to be positive and the inward negative. \textbf{Right:} Equilateral Triangle mapped to circle and  regular, triangular pyramid mapped to the unit sphere using harmonic energy minimizing maps}
    \label{fig:directions}
\end{figure}

\subsection{Relation to the 1D Problem}\label{sec:1dProblem}

The numerical scheme presented in Section \ref{sec:opt} can also be adapted to map the Vietoris-Rips complex $VR(X)$ on a data set $X$ to the unit circle $S^1$, recovering the original results of \cite{de2011persistent}. (Note that since $S^1$ has no higher homotopy groups, an initial map can be constructed on the entirety of $VR(X)$, as in \cite{de2011persistent}.) For harmonic energy, it is more efficient to explicitly calculate the harmonic representative as in \cite{de2011persistent}, but formulating the 1D problem is useful for building intuition for the 2D case. Here, the harmonic energy and spring energy for $f \colon VR(X) \to S^1$ are given by:

\begin{equation}\label{eqn:1dNRGS}
    \mathcal{E}_H (f) = \sum_{x \in VR(X)^1} \frac{1}{2}||L(f (x))||^2 
    \quad \text{and} \quad
    \mathcal{E}_S (f) = \sum_{x \in VR(X)^1} \frac{1}{2}||k(L(f (x))-R)||^2   
\end{equation}
where $L(f (x))$ is the arc length of $f(x)$, $k$ is a spring constant and $R$ is the spring rest-length. Similar to the set up for the spherical case, we interpret winding numbers $w =\pm 1$ as wrapping an edge around the circle in a clockwise ($w= 1$) or anti-clockwise ($w= -1$) direction. Since there is a unique harmonic representative of $f$ we do not need to compute M{\"o}bius updates, and can instead directly minimize the energy using a first-order scheme. The arguments presented in Section \ref{sec_preserving_homotopical_info} can be adapted to show that this scheme preserves the homotopy class of the map, and numerical results in Section \ref{sec:results1d} showcase the usefulness of the spring energy for 1D problems.

\section{Energy Minimization Preserves Homotopical Information} \label{sec_preserving_homotopical_info}

The purpose of this section is to fill in the details of Steps (7)-(9) of Section \ref{subsec_main_pipeline}. Beginning with a map $\tilde{\alpha} \colon \sk_3(X) \to S^2$, i.e., a representative of a homotopy class in $[\sk_3(X), S^2]$, the goal is to find a smoother but homotopic map $\tilde{\omega} \colon \sk_3(X) \to S^2$. As described in Section \ref{subsec_main_pipeline}, the minimization procedure begins by letting $\alpha_0$ be the restriction of $\tilde{\alpha}$ to $\sk_2(X)$ and then produces a sequence of maps $\alpha_i\colon\sk_2(X) \to S^2$ of lower harmonic (or spring) energy, terminating at a final map $\alpha_N = \omega$. It is crucial, then, that we verify each map produced by our minimization procedure extends to $\sk_3(X)$ and that this extension is homotopic to the previous so that, ultimately, $\tilde{\alpha}$ and $\tilde{\omega}$ represent the same homotopy class of maps $\sk_3(X) \to S^2$. Without this, there would be no topological connection between $\tilde{\alpha}$, representing a significant feature of the data, and the map $\tilde{\omega}$ produced by our algorithm. 

Our strategy is to, first, show that $\alpha_i \simeq \alpha_{i+1}$ for all $0 \leq i \leq N-1$ as maps $\sk_2(X) \to S^2$. This is the content of Theorem \ref{thm_minimization_preserves_homotopy_on_2_skeleton}. Then, Theorem \ref{thm_extension_to_3_skeleton} extends this sequence of homotopies to $\sk_3(X)$. Towards this end we begin with the following, which applies to any topological space $Y$ (though our main interest is the case when $Y$ is a simplicial complex).

\begin{prop}\label{prop_homotopic_if_not_antipodal}
Let $Y$ be a topological space, and suppose $f$ and $g$ are two maps $f, g \colon Y \to S^2$ from Y to the unit sphere in $\RR^3$. Let $\|-\|$ denote the Euclidean norm and suppose $\|f(y) - g(y) \| < 2$ for all $y \in Y$. Then $f \simeq g$.
\end{prop}
\begin{proof}
Since $\|f(y) - g(y)\|<2$ for all $y \in Y$, we know that $f(y)$ and $g(y)$ are never antipodal, in which case the straight-line homotopy
\begin{align}
f_t(y) = \dfrac{(1-t)f(y) + tg(y)}{\norm{(1-t)f(y) + tg(y)}}
\end{align}
is well-defined, continuous, and shows $f \simeq g$.
\end{proof}
Proposition \ref{prop_homotopic_if_not_antipodal} lets us handle the gradient descent steps of the minimization procedure as follows.

\begin{prop}\label{prop_homotopic_if_gradient_descent_step}
Suppose $\alpha_i$ and $\alpha_{i+1}$ are successive functions in the minimization procedure described in Section \ref{sec_smoothing}. If $\alpha_{i+1}$ is obtained from $\alpha_i$ by a gradient descent step, then $\alpha_i \simeq \alpha_{i+1}$ as maps $\sk_2(X) \to S^2$.
\end{prop}

\begin{proof}
By limiting the distance that the vertices and barycenter of each 2-simplex are moved by each gradient descent step, this follows by taking $Y = \sk_2(X)$ as in Proposition \ref{prop_homotopic_if_not_antipodal}.
\end{proof}

Next, we must show that composing with a M{\"o}bius transformation preserves the homotopy class of a map $\sk_2(X) \to S^2$. We suspect the following is well-known to experts, but record a proof here for the sake of completeness.

\begin{prop}\label{prop_mobius_transformations}
Identifying $S^2$ with the Riemann sphere (or extended complex plane), any M{\"o}bius transformation $f$ is homotopic to the identity.
\end{prop}

\begin{proof}
There are a number of ways to see this, but one of the more straightforward is as follows. Identify $S^2$ with the extended complex plane and represent $f$ as $f(z) = \dfrac{az+b}{cz+d}$ for $a,b,c,d \in \CC$ with $ad-bc \neq 0$. If $c \neq 0$, decompose $f$ into a composition of simpler M{\"o}bius functions
\begin{align}
f_1(z) = z + \frac{c}{d}, \quad f_2(z) = \frac{1}{z}, \quad f_3(z) = \frac{bc-ad}{c^2}z, \quad f_4(z) = z + \frac{a}{c}
\end{align}
whence $f = f_4\circ f_3 \circ f_2 \circ f_1$. For $i = 1, 3, 4$, each $f_i$ is a polynomial on $\CC$ that we have canonically extended to $\CC \cup \sett{\infty}$. It is a well-known fact that the degree of such a function viewed as a map $S^2 \to S^2$ is equal to its degree as a polynomial. Therefore, $\deg(f_i) = 1$ and so $f_i$ is homotopic to the identity for $i = 1, 3, 4$ by the Hopf degree theorem. Lastly, $f_2$ represents a rotation $S^2$ by 180 degrees, so is also homotopic to the identity. Thus, $f$ is a composition of functions each of which is homotopic to the identity, so $f$ is as well. The case $c = 0$ is handled similarly.
    \end{proof}

Together, Propositions \ref{prop_homotopic_if_gradient_descent_step} and \ref{prop_mobius_transformations} show that our minimization procedure preserves the homotopy class of $\alpha_0 \colon \sk_2(X) \to S^2$.

\begin{thm}\label{thm_minimization_preserves_homotopy_on_2_skeleton}
Suppose $\alpha_0 \rightsquigarrow \alpha_1 \rightsquigarrow \alpha_2 \rightsquigarrow \ldots \rightsquigarrow \alpha_N$ is the sequence of maps $\sk_2(X) \to S^2$ obtained via the minimization procedure described in Section \ref{sec_smoothing}. Then $\alpha_0$ $\simeq \alpha_N$.
\end{thm}
\begin{proof}
It follows from Proposition \ref{prop_homotopic_if_gradient_descent_step} that $\alpha_0 \simeq \alpha_1$. Inductively, suppose that we have shown $\alpha_{i-1} \simeq \alpha_i$. If $\alpha_{i+1}$ is obtained from $\alpha_i$ by a gradient descent step, then Proposition \ref{prop_homotopic_if_gradient_descent_step} also shows $\alpha_i \simeq \alpha_{i+1}$. Otherwise, $\alpha_{i+1} = m \circ \alpha_i$ for a M{\"o}bius transformation $m$. In this case, it follows from Proposition \ref{prop_mobius_transformations} that
\begin{align}
\alpha_{i+1} = m \circ \alpha_i \simeq \id \circ \ \alpha_i = \alpha_i
\end{align}
which completes the proof.
\end{proof}

We now show that the sequence of homotopies in Theorem \ref{thm_minimization_preserves_homotopy_on_2_skeleton} extends to $\sk_3(X)$.
\begin{thm}\label{thm_extension_to_3_skeleton}
Each map $\alpha_i$ in Theorem \ref{thm_minimization_preserves_homotopy_on_2_skeleton} extends to a map $\tilde{\alpha_i}\colon\sk_3(X) \to S^2$ and $\tilde{\alpha_i} \simeq \tilde{\alpha_{i+1}}$ for all $0 \leq i \leq N-1$.
\end{thm}
\begin{proof}
By construction, an extension of $\alpha_0$ to $\sk_3(X)$ is given by $\tilde{\alpha}$. Then, inductively for $0 \leq i \leq N-1$, the Homotopy Extension Property of the pair $(\sk_3(X), \sk_2(X))$ gives an extension of $\alpha_{i+1}$ to $\tilde{\alpha_{i+1}}\colon\sk_3(X) \to S^2$ and shows that $\tilde{\alpha_i} \simeq \tilde{\alpha_{i+1}}$ as maps $\sk_3(X) \to S^2$.
\end{proof}

\section{Numerical Results}\label{sec:numerical}

Here we present some numerical results on syntheic and real-word data. All experiments were performed in Python using Ripser \cite{ripser, ctralie2018ripser} to compute cohomology classes. %The code and user parameters will be publicly released with the formal submission of this paper. 

\subsection{Synthetic Problems}\label{sec:results1d}

We begin with a series of illustrative tests on the 1D problem to demonstrate the efficacy of our approach and showcase the difference between the harmonic and spring energy. 
%In these cases, we can also verify that the process of minimizing the energies is null-harmonic by examining the \textcolor{red}{shell of $\alpha*-\tilde\alpha$}. Note that if the shell does not cover any point in the sphere then it is null-homotopic.\textcolor{red}{Niko cite homotopy thing} \textcolor{red}{NOT DOING THIS ANY MORE}
Figure \ref{fig:1D synthetic} shows five tests of our approach on data sampled from topological circles. The first column shows the original data, the second plots the cocycles computed by Ripser, the third shows the final mapping of the data, and the fourth shows a plot of the circular coordinate ($x$-axis) vs. an arc length parameterization coordinate of the data. The first test is an evenly sampled circle, for which we recover the original coordinates almost exactly. The second shows a much sparser sampling from a noisy circle. The third row is a trefoil knot, embedded in three dimensions, but visualized here in two. 

The fourth and fifth rows in Figure \ref{fig:1D synthetic} demonstrate the difference between spring and harmonic energies by applying them to the same problem. We begin with an ellipse in 2D which is sampled at a rate proportional to its curvature and then embedded isometrically in 50 dimensions via a random orthogonal projection. Next we add Gaussian noise in the embedded space. This type of sampling is common in both theoretical work and applications \cite{belkin2003laplacian, cayton2005algorithms, niyogi2008finding, schonsheck2022semi}. Next, we obtain the cocycles with Ripser and then minimize harmonic energy (shown in the fourth row) and spring energy (shown in the fifth). Here we see that the spring energy produces a much more reasonable map since it is able to better distribute points from the more densely sampled regions around the ends of the ellipse. By spreading out these regions of high curvature, we obtain a map which is both more visually readable and useful for downstream tasks.

\begin{figure}[H]
    \centering
    \includegraphics[width=.8\linewidth]{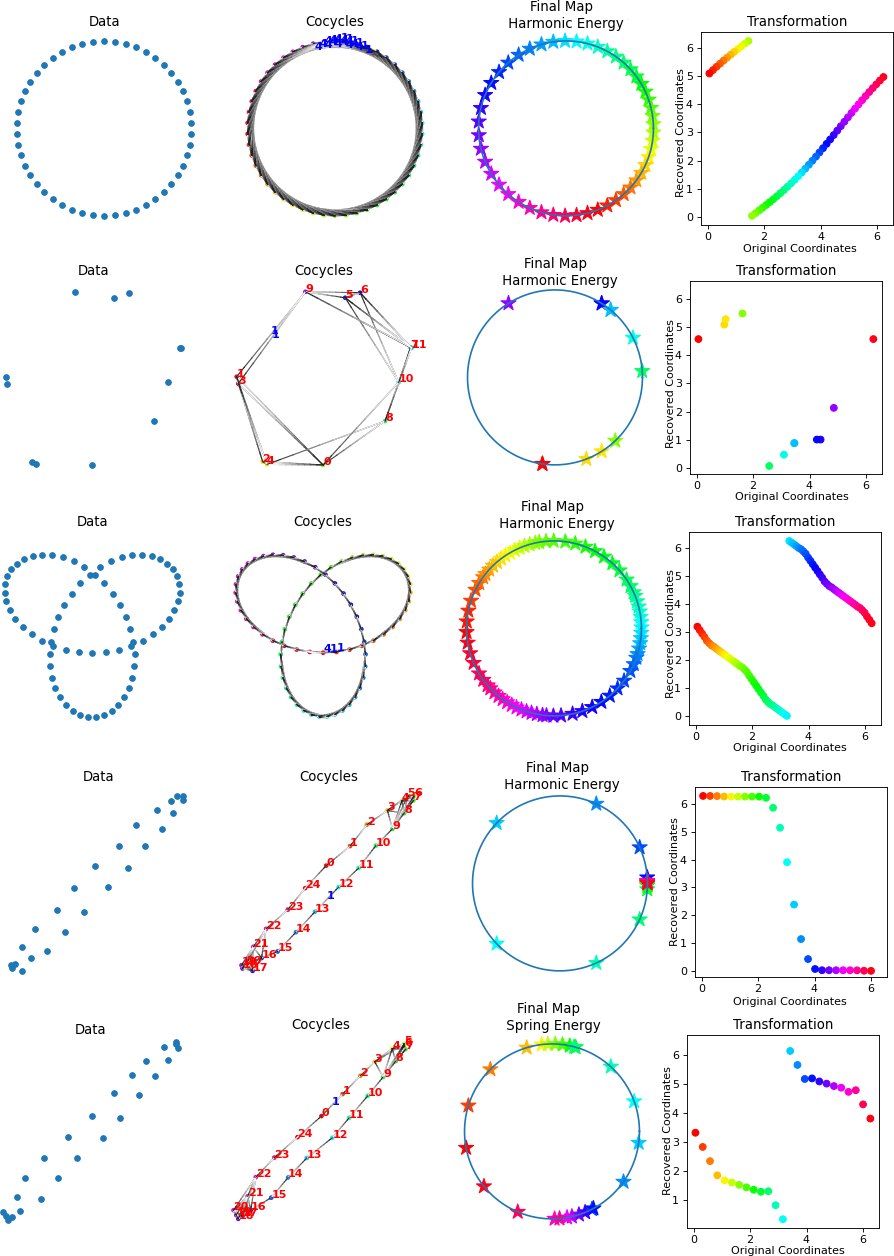}
    \caption{\textbf{First Row:} Evenly sampled circle, \textbf{Second Row:} Noisy, sparsely sampled circle, \textbf{Third Row:} Trefoil knot in 3D (visualized in 2D) with spring energy, \textbf{Fourth row:} Noisy Ellipse in 50D (visualized in 2D without noise) with harmonic energy, \textbf{Last row:} Same as above but with spring energy}
    \label{fig:1D synthetic}
\end{figure}

Next, we demonstrate our approach to problems using degree two cohomology. Figure \ref{fig:2D synthetic} shows a battery of tests. The first column shows four visualizations of the original data from different angles. We then embed in $\RR^{50}$ by multiplying the original coordinates with a random $(3 \times 50)$ matrix with orthogonal columns to simulate a low-dimensional object embedded in a higher dimensional space. In the noisy cases, we add noise in this ambient dimension. The second column shows the barcodes detected by Ripser. We highlight the selected feature we use as input to the variational model in red. The third column shows the final  coordinates produced by our algorithm. Finally, the fourth column shows the generated azimuth and elevation coordinates (measured in radians) of the initial data versus the final coordinates. To measure this we compute a rigid alignment of the final coordinates produced by our algorithm using their known correspondences via singular value decomposition (see \cite{Dubrovina2010} for additional details).  Note that, below, the $x$ and $y$ axes of the ``Azimuth'' and ``Elevation'' diagrams are ``circular,'' so that points very near the top (resp. right) are, in fact, quite close to points near the bottom (resp. left) and vice-versa. For the ellipsoid cases, we use the ParamTools package \cite{Kapur2020} to generate the sampling.

\begin{figure}[H]
    \centering
    \includegraphics[width=\linewidth]{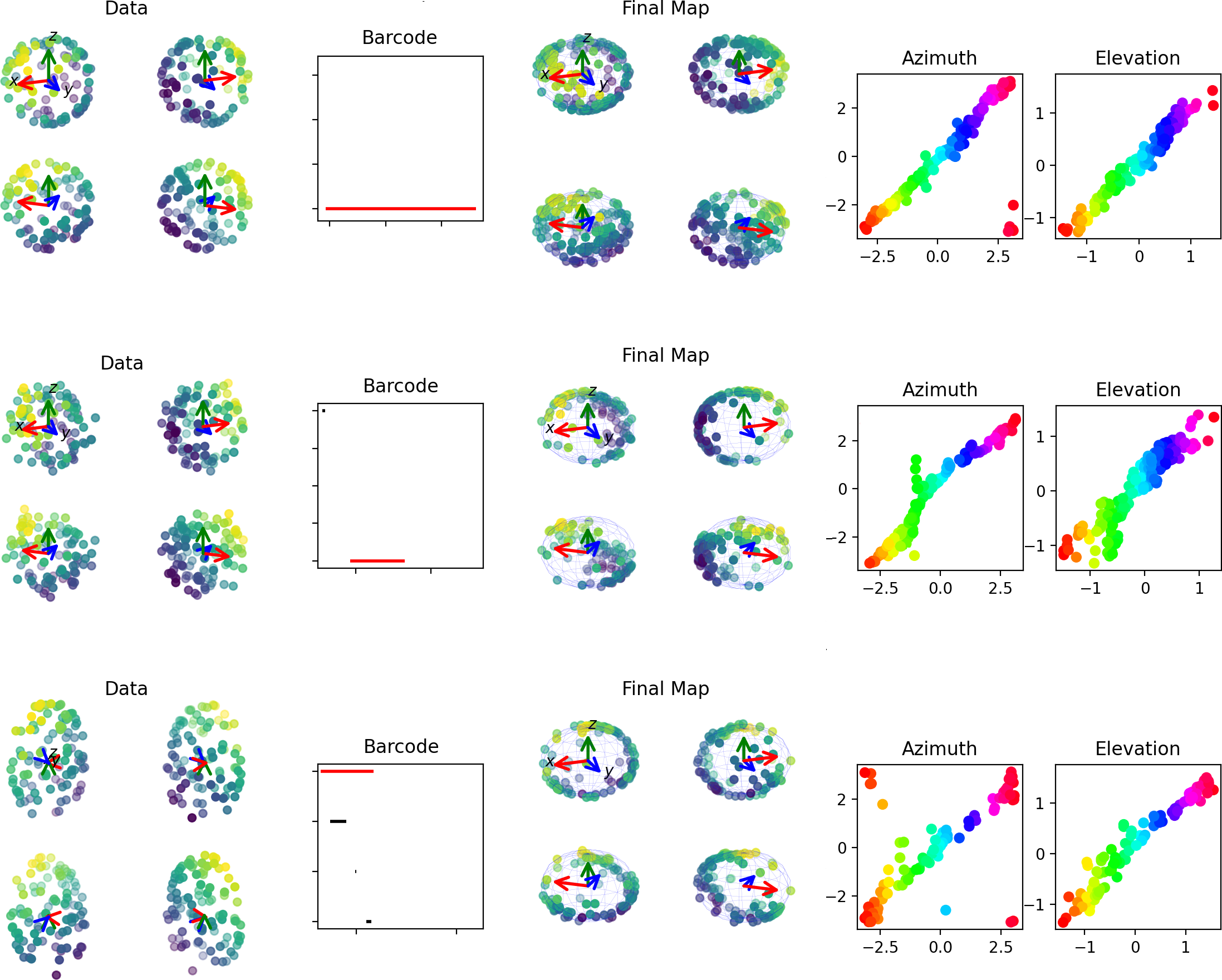}
    \caption{\textbf{First Row}: Sphere \textbf{Second Row}: Sphere with noise, \textbf{Third Row}: Ellipse}
    \label{fig:2D synthetic}
\end{figure}

From these experiments we conclude that our method is able to define spherical coordinates for point clouds formed by sampling spheres and ellipsoids embedded in high dimensions. After alignment, we observe that the coordinates of the points are close to the ground truth coordinates used to generate the underlying data. Additionally, by tuning the spring constants we can produce visually pleasing maps which are useful for data visualization and downstream tasks to be explored in future research.

\subsection{Feature Selection}\label{subsec_feature_selection}

While we have chosen high-persistence bars as the inputs to our pipeline in the experiments above, our methodology can be used to derive spherical coordinates from any chosen bar. However, if that bar does not represent a significant feature of the data, then the coordinates we find are not likely to be meaningful. Figure \ref{fig:feature_selection} illustrates this for the 1D and 2D problems.

\begin{figure}[H]
    \centering
    \includegraphics[width=\linewidth]{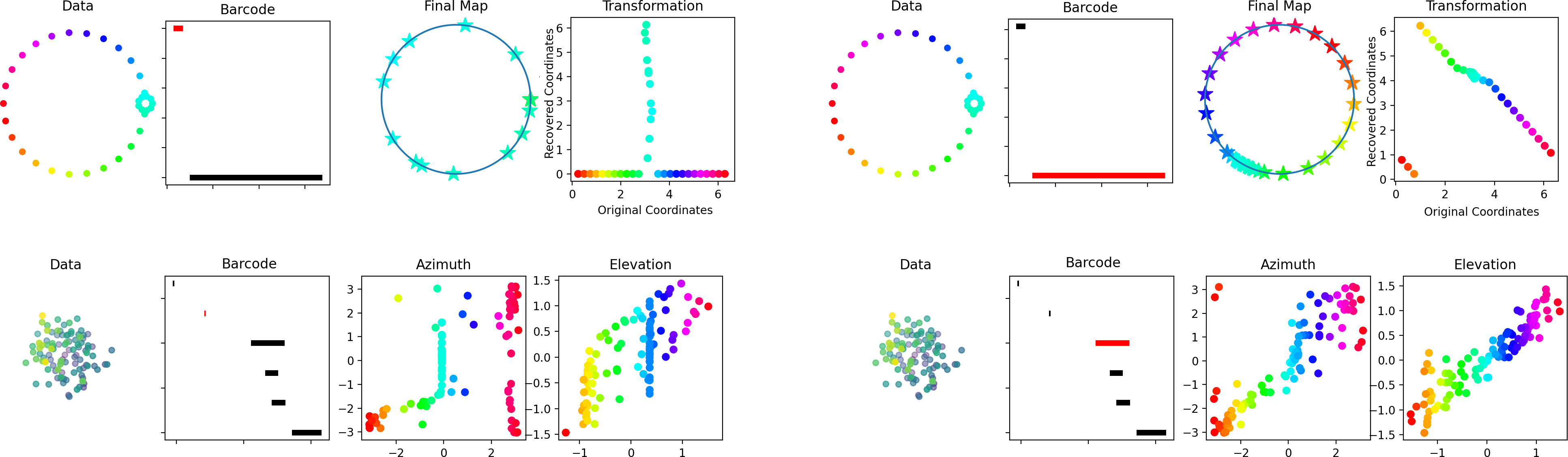}
    \caption{\textbf{First Row}: Results for two-circle example. The first four images are associated with the shortest bar and the last four with the longest. \textbf{Second Row}: Results for a noisy sphere. Again the first four images are associated with the shortest bar and the last four with the longest.}
    \label{fig:feature_selection}
\end{figure}

First, in the top row, we consider a data set consisting of samples from a large circle, augmented by samples from a smaller circle centered around a point on the circumference of the first circle. If we choose a low-persistence bar, then the cocycles we compute are associated with the smaller circle, and the final coordinates produced by our algorithm do not capture the global topology of the data set. The first image shows the data set and the second image shows the barcode diagram with the selected feature highlighted in red. The third image shows the final embedding and the fourth image displays the coordinate accuracy. These results do not capture the main structure of the data set. However, we can recover a good embedding by choosing the highest-persistence bar (as shown in the second four images in the top row).

A similar phenomenon happens when working with noisy data. In the second row, we illustrate two tests on a very noisy sphere. Again, in the first test, we chose the smallest bar (indicated by the red line in the barcode) and find the results unsatisfactory. However, by selecting the longest bar we find good representations. Of course, for data sets with multi-scale features, it may be useful to develop multiple embeddings. In the next section, we study some simple cases but leave a more thorough study of the problem for future works.

\subsection{More Complicated Topology}

In this section, we illustrate several interesting cases where the underlying topology is not simply a single sphere, but rather some combination of two spherical sub-domains. We show how pruning simplices in a post-process can improve the results and effectively disentangle the subdomains allowing us to develop two maps, one for each feature. We begin with two disconnected spheres, which we color red and blue to help disambiguate the points in the embedding. Of course, these could be separated in pre-possessing, but, for now, we delay this in order to build intuition. 

We begin by applying our algorithm to this this data set exactly as we did in the previous experiments. In the top row of Figure \ref{fig:2ds_con}, we select the feature associated with the red sphere and in the bottom row we select the feature associated with the blue sphere. Even though the spheres are disconnected, the points on the blue sphere still contribute to the zero center of mass constraint and result in an unevenly distributed final map (left hand side of Figure \ref{fig:2ds_con}). However, after pruning these points manually we are able to recover a good representation of each sphere, as shown in the right hand side of Figure \ref{fig:2ds_con}. From this observation, we propose a pruning scheme: After completion of the algorithm, if the results are unsatisfactory, we remove all simplices which have an area less than some threshold (in this case $1e^{-2}$) and apply the minimization procedure again to the remaining points. 
%\textcolor{red}{It seems like, in practice, you would never actually know if the results are unsatisfactory. So the recommendation would be to just always do the pruning procedure? For future work, we could show that doing this in cases where it's not necessary like if you just have one sphere, does not screw up the results.}

\begin{figure}[H]
    \centering
    \includegraphics[width=\linewidth]{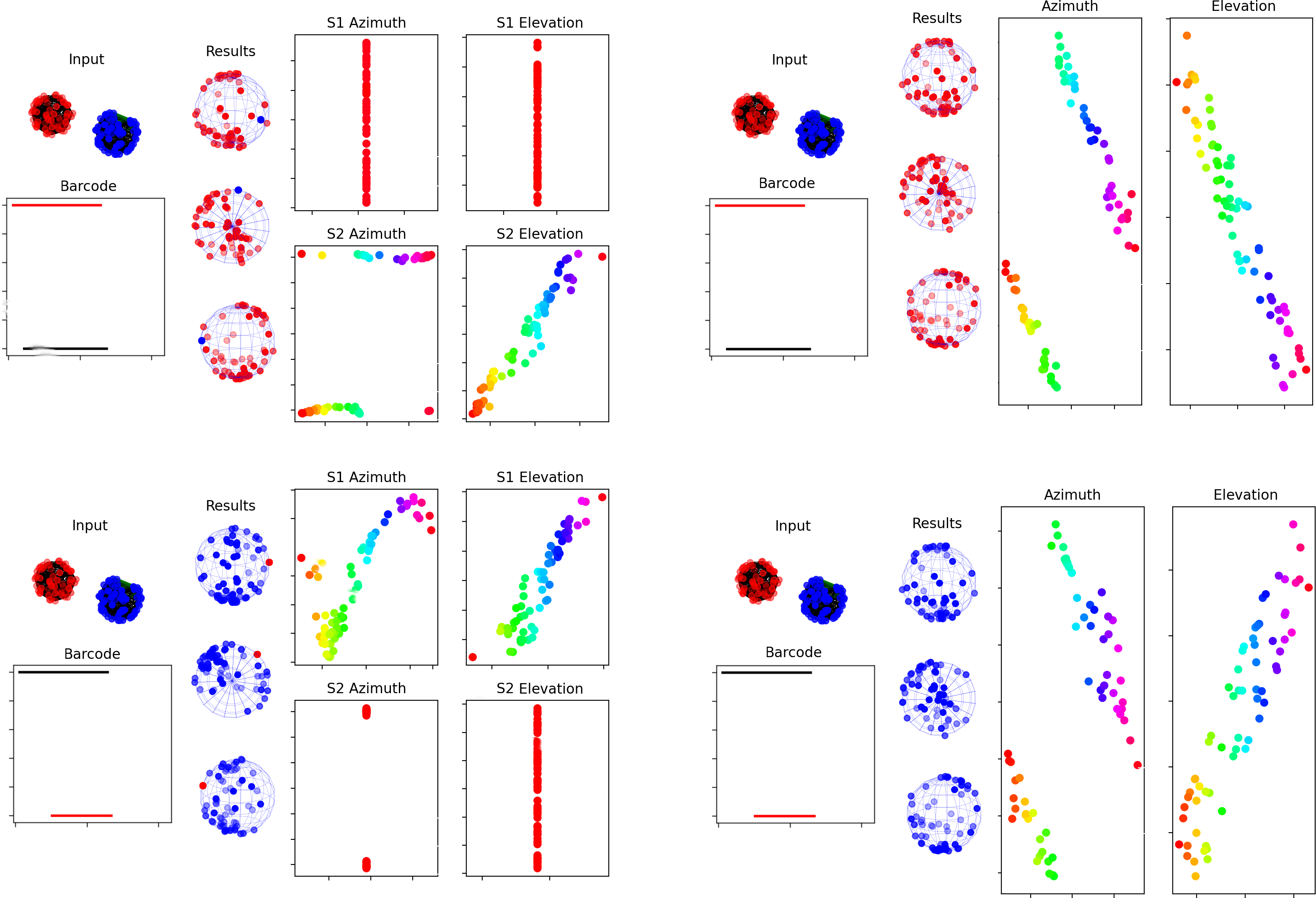}
    \caption{\textbf{Column 1 of each subfigure}: Top: Disconnected input data. Simplex connections produced by Ripser are shown as black lines. Bottom: Barcode with the selected bar highlighted in red. \textbf{Column 2 of each subfigure}: Final maps visualized from three angles \textbf{Columns 3 and 4 of each subfigure}: Computed coordinates vs ground truth.}
    \label{fig:2ds_con}
\end{figure}

A similar phenomenon appears when applying our algorithm naively to a data set consisting of points sampled from a wedge of two spheres (Figure \ref{fig:2ds_dis}). Here there are two main topological features (i.e., two high-persistence bars in the barcode), corresponding to the two spheres. Neither feature produces a satisfactory final map since the underlying data is non-spherical (as shown in the left column of Figure \ref{fig:2ds_dis}). Note that in the first case, the chosen feature corresponds to the red sphere, and in the second the chosen feature corresponds to the blue sphere. In both of these cases, we see that the points associated with the sphere not chosen are mapped very close to a lower dimensional subspace (a line) by the algorithm. From this observation we propose to apply the pruning scheme explained above to improve the results. Figure \ref{fig:2ds_dis} shows the results before (left column) and after (right column) applying our pruning technique. Although the results are not perfect (note that there are still some red points on the blue sphere and vice versa) we observe that the results are much improved. Recently, several works have aimed at organizing or disentangling topological features from data sets with noisy barcodes \cite{perea_circular, perea_projective, chen2019topological}. We hope to explore this pruning technique in concert with these methods to develop procedures that are robust to noise in future works.

\begin{figure}[H]
    \centering
    \includegraphics[width=.95\linewidth]{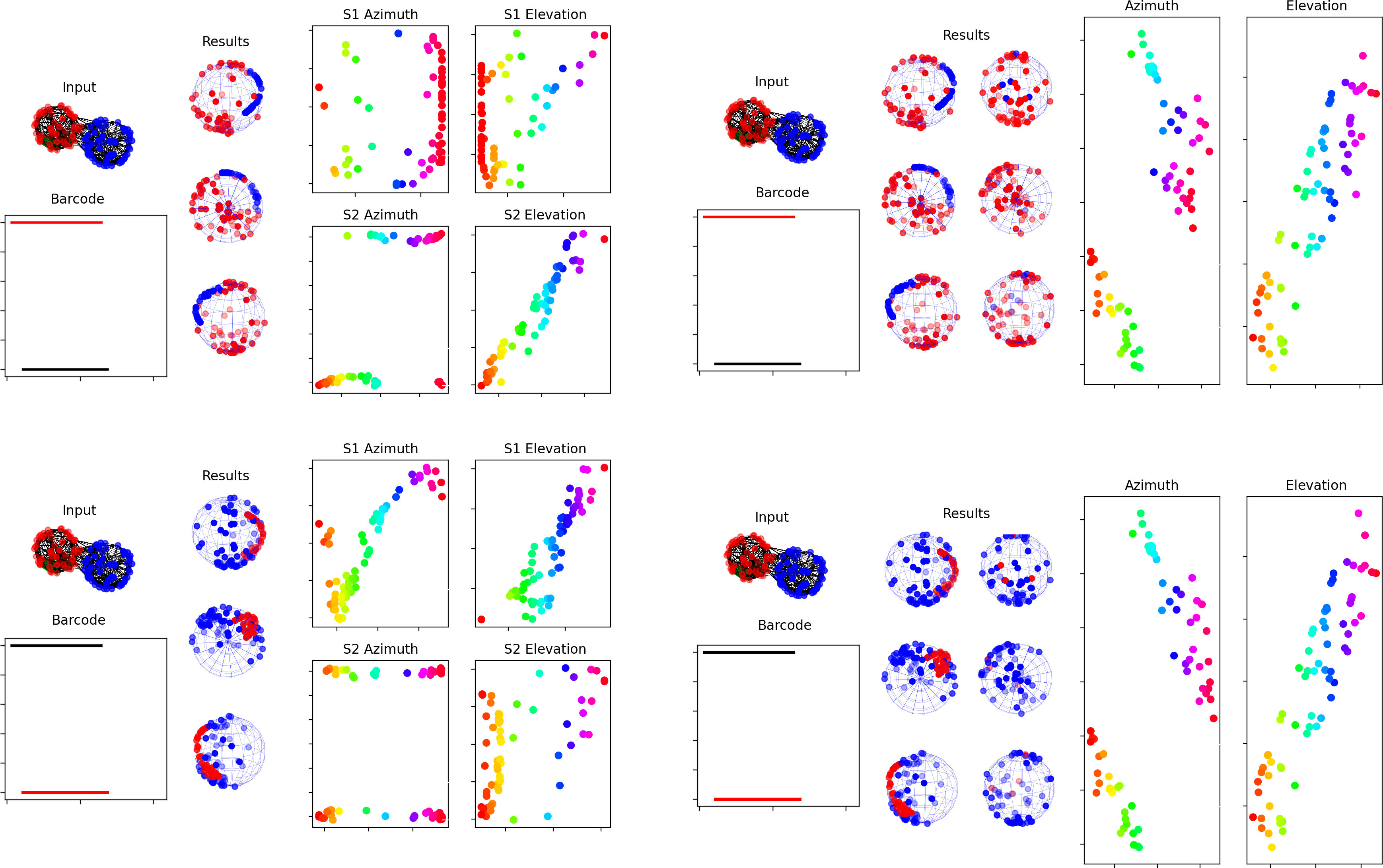}
    \caption{\textbf{Column 1 of each subfigure}: Top: Input data. Simplex connections produced by Ripser are shown as black lines. Bottom: Barcode with the selected feature highlighted in red. \textbf{Column 2 of each subfigure}: Final maps visualized from various angles. \textbf{Columns 3 and 4 of each subfigure}: Computed coordinates vs ground truth.}
    \label{fig:2ds_dis}
\end{figure}

\subsection{Applied Topological Data Analysis Problems}

In this section we apply our techniques to a few topological data analysis problems. Each of the following examples comes from a single well-defined circular or spherical domain, and as a result we do not employ any post-processing.
We begin with a 1D example taken from the Columbia Object Image Library (COIL-100) data set \cite{nene1996columbia}. The data set contains images of everyday objects (in our case a rubber ducky) photographed from different angles along a single axis of rotation. There are 72 images per object (taken at 5 degree intervals) and each image has a resolution of 128 by 128 and therefore can be represented as a point in $\RR^{16384}$ by flattening the image into a single vector. These points can be modeled as a discretization of a topological circle in $\RR^{16384}$. Applying our algorithm to the problem, we are able to closely recover the angle at which each picture was taken. Figure \ref{fig:duck} shows the embedding along with the coordinate reconstruction using both harmonic and spring energy minimization. For this case, the harmonic energy minimization did not produce satisfactory results as shown in the left side of Figure \ref{fig:duck}. This is due to the irregularity of the 3D model. Although each image is offset by 5 degrees, the pixel distance between the different images varies greatly and as a result the harmonic energy minimization leads to a clumping of the data. However, as shown on the right side of Figure \ref{fig:duck} we can still obtain a more evenly distributed map via the spring energy formulation.

\begin{figure}[H]
    \centering
    \includegraphics[width=\linewidth]{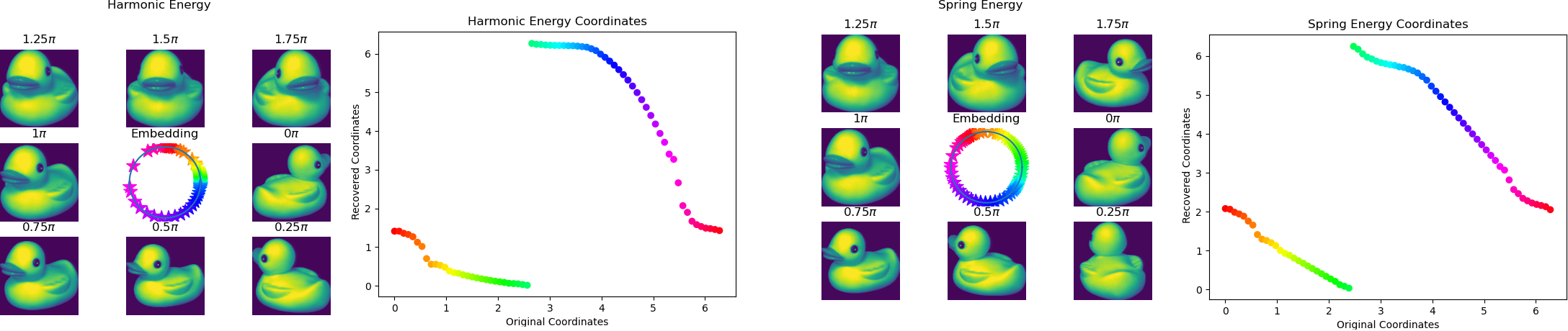}
    \caption{\textbf{Column 1 of each subfigure}: The embedding of the COIL-100 duck. Starting with an arbitrary origin, we show the images whose final coordinates differ by $\frac{\pi}{4}$. \textbf{Column 2 of each subfigure}: Coordinates of final map vs. ground truth coordinates from data set} 
    \label{fig:duck}
\end{figure}

We conduct a similar experiment to the COIL problem by creating our own data set of images based on rotating an object in 3D. In Figure \ref{fig:bunny} we show the results of applying our algorithm to a data set consisting of 200, 128 by 128 pixel images of the Stanford Bunny \cite{pomerleau2012challenging} under random rotations. Although the barcode of the data set is irregular, by selecting the longest bar and using a large enough spring constant we are able to embed the images onto the sphere in a meaningful way and recover the approximate angle each photo was taken from.

\begin{figure}[H]
    \centering
    \includegraphics[width=\linewidth]{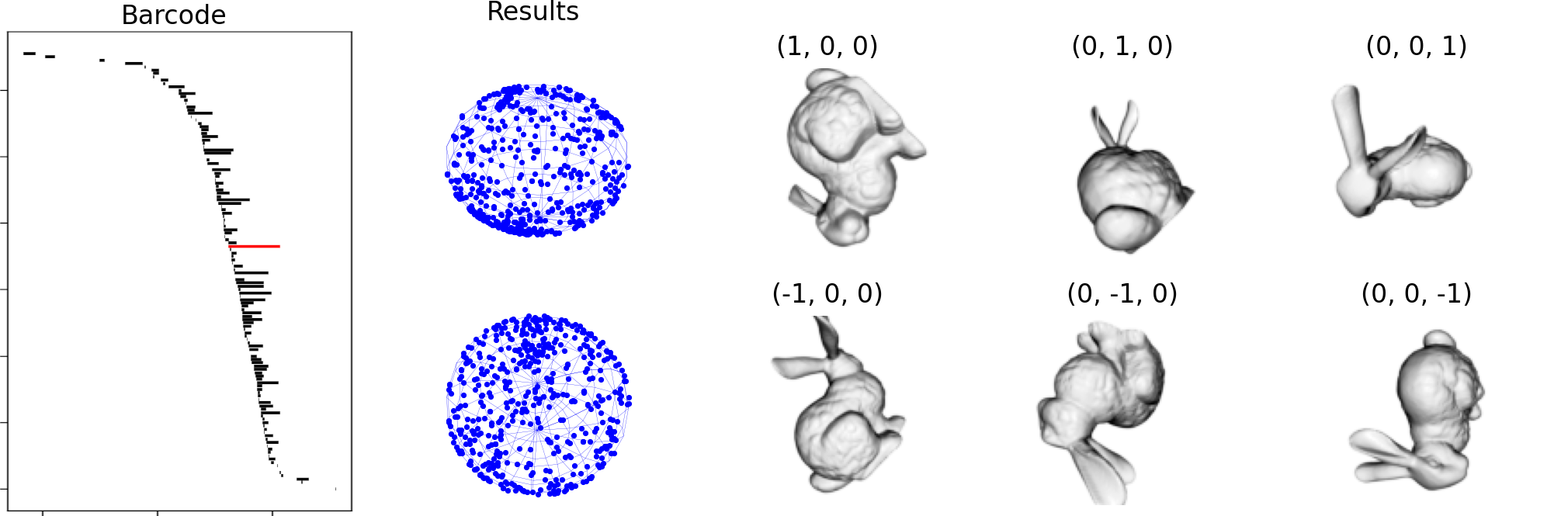}
    \caption{Stanford bunny experiment. In the right three columns, we show the images nearest to the point indicated.}
    \label{fig:bunny}
\end{figure}

We also conduct a synthetic experiment inspired by applications of topological data analysis to neuroscience. It has long been known \cite{place_cells} that certain cells in the hippocampus, called place cells, have spatially localized receptive fields. In other words, they represent one's position during navigation by firing rapidly at specific locations in an environment, while maintaining low activity at all others. Analysis of such data using persistent (co)homology can recover the coarse geometry of the environment itself \cite{Clique_topology, decoding_neural}. Similarly, head direction as encoded by a product of circles can be recovered \cite{Finkelstein_three_dimensional_coding} from neural activity using persistence techniques. We model a variation of this idea: what if place cells were embedded not on, e.g., a circle as a rat runs around a circular track, but rather on a sphere?

We begin by choosing 64 nearly-equally distributed points on the sphere using the Fibonacci lattice \cite{stanley1975fibonacci, wiki:Young_Fibonacci_lattice}. These points serve as sensors for a discrete random walk on the sphere. During the $i^{th}$ step a sensor $j$ measures the distance to a walk point $x_i$ as $S_j(x_i) = \exp{(-D(x_i, j))} + \mathcal{N}$ where $D(x_i,j)$ is the distance between $x_i$ and the $j^{th}$ sensor and $\mathcal{N}$ is random Gaussian noise. We simulate 25 discrete random walks of length 25 steps and geodesic step size $\delta=0.1$ on the sphere to generate a matrix $S=s_{ij}$. Then our goal is to use $S$ to recover the original coordinates of the sensors. Figure \ref{fig:sensors} shows the coordinate recovery for the cases with $\mathcal{N}$ having standard deviation $\sigma = 0,0.2,0.5$. From these experiments we conclude that our method is fairly robust to noise. We look forward to conducting similar experiments on real-word sensor data in future works.

\begin{figure}[H]
    \centering
    \includegraphics[width=\textwidth]{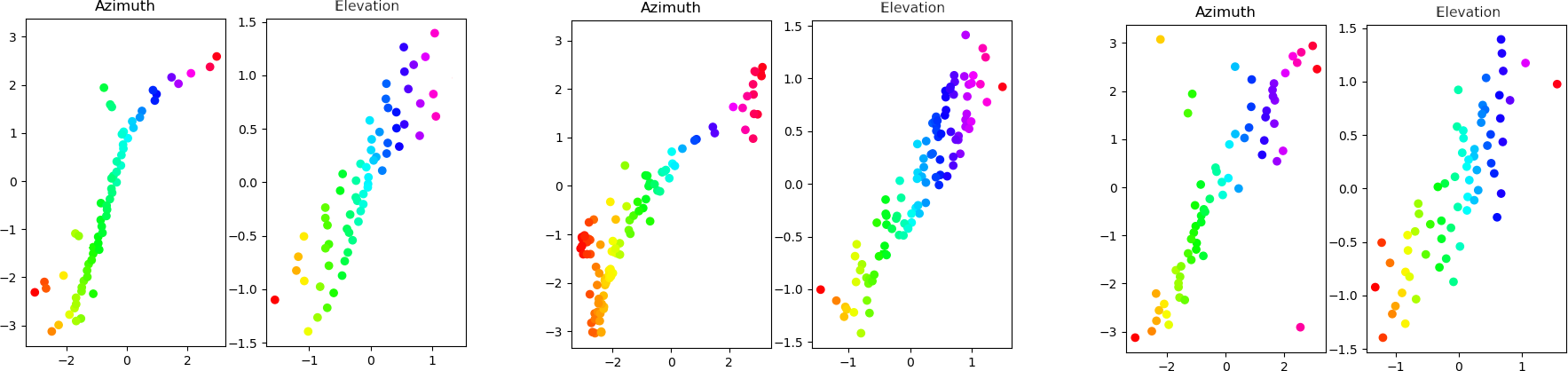}
    \caption{Results from the sensors experiments under different levels of noise. From right to left the noise values are: $\sigma=0,0.2,0.5$}
    \label{fig:sensors}
\end{figure}

\section*{Declarations}
The authors have no relevant financial or non-financial interests to disclose.






\bibliographystyle{plain}
\bibliography{bib}% common bib file

\begin{thebibliography}{10}

\bibitem{baden2018mobius}
Alex Baden, Keenan Crane, and Misha Kazhdan.
\newblock M{\"o}bius registration.
\newblock In {\em Computer Graphics Forum}, volume~37, pages 211--220. Wiley
  Online Library, 2018.

\bibitem{balasubramanian2002isomap}
Mukund Balasubramanian and Eric~L Schwartz.
\newblock The isomap algorithm and topological stability.
\newblock {\em Science}, 295(5552):7--7, 2002.

\bibitem{ripser}
Ulrich Bauer.
\newblock Ripser: efficient computation of {V}ietoris-{R}ips persistence
  barcodes.
\newblock {\em Journal of Applied and Computational Topology}, 2021.

\bibitem{belkin2003laplacian}
Mikhail Belkin and Partha Niyogi.
\newblock Laplacian eigenmaps for dimensionality reduction and data
  representation.
\newblock {\em Neural computation}, 15(6):1373--1396, 2003.

\bibitem{bonnabel2013stochastic}
Silvere Bonnabel.
\newblock Stochastic gradient descent on riemannian manifolds.
\newblock {\em IEEE Transactions on Automatic Control}, 58(9):2217--2229, 2013.

\bibitem{Carlsson_topology_and_data}
Gunnar Carlsson.
\newblock Topology and data.
\newblock {\em Bull. Amer. Math. Soc. (N.S.)}, 46(2):255--308, 2009.

\bibitem{cayton2005algorithms}
Lawrence Cayton.
\newblock Algorithms for manifold learning.
\newblock {\em Univ. of California at San Diego Tech. Rep}, 12(1-17):1, 2005.

\bibitem{chazal_structure_and_stability}
Fr\'{e}d\'{e}ric Chazal, Vin de~Silva, Marc Glisse, and Steve Oudot.
\newblock {\em The structure and stability of persistence modules}.
\newblock SpringerBriefs in Mathematics. Springer, [Cham], Switzerland, 2016.

\bibitem{chen2019topological}
Chao Chen, Xiuyan Ni, Qinxun Bai, and Yusu Wang.
\newblock A topological regularizer for classifiers via persistent homology.
\newblock In {\em The 22nd International Conference on Artificial Intelligence
  and Statistics}, pages 2573--2582. PMLR, 2019.

\bibitem{connor2020representing}
Marissa Connor and Christopher Rozell.
\newblock Representing closed transformation paths in encoded network latent
  space.
\newblock In {\em Proceedings of the AAAI Conference on Artificial
  Intelligence}, volume~34, pages 3666--3675, 2020.

\bibitem{davidson2018hyperspherical}
Tim~R Davidson, Luca Falorsi, Nicola De~Cao, Thomas Kipf, and Jakub~M Tomczak.
\newblock Hyperspherical variational auto-encoders.
\newblock {\em arXiv preprint arXiv:1804.00891}, 2018.

\bibitem{de2011persistent}
Vin De~Silva, Dmitriy Morozov, and Mikael Vejdemo-Johansson.
\newblock Persistent cohomology and circular coordinates.
\newblock {\em Discrete \& computational geometry}, 45(4):737--759, 2011.

\bibitem{Dubrovina2010}
Anastasia Dubrovina, Leonidas Guibas, and Hao Su.
\newblock Cs468: Machine learning for 3d data, February 2010.

\bibitem{falorsi2018explorations}
Luca Falorsi, Pim De~Haan, Tim~R Davidson, Nicola De~Cao, Maurice Weiler,
  Patrick Forr{\'e}, and Taco~S Cohen.
\newblock Explorations in homeomorphic variational auto-encoding.
\newblock {\em arXiv preprint arXiv:1807.04689}, 2018.

\bibitem{Finkelstein_three_dimensional_coding}
Arseny Finkelstein, Dori Derdikman, Alon Rubin, Jakob~N. Foerster, Liora Las,
  and Nachum Ulanovsky.
\newblock Three-dimensional head-direction coding in the bat brain.
\newblock {\em Nature}, 517(7533):159--164, 2015.

\bibitem{Clique_topology}
Chad Giusti, Eva Pastalkova, Carina Curto, and Vladimir Itskov.
\newblock Clique topology reveals intrinsic geometric structure in neural
  correlations.
\newblock {\em Proceedings of the National Academy of Sciences of the United
  States of America}, 112, 02 2015.

\bibitem{gu2004genus}
Xianfeng Gu, Yalin Wang, Tony~F Chan, Paul~M Thompson, and Shing-Tung Yau.
\newblock Genus zero surface conformal mapping and its application to brain
  surface mapping.
\newblock {\em IEEE transactions on medical imaging}, 23(8):949--958, 2004.

\bibitem{gu2003global}
Xianfeng Gu and Shing-Tung Yau.
\newblock Global conformal surface parameterization.
\newblock In {\em Proceedings of the 2003 Eurographics/ACM SIGGRAPH symposium
  on Geometry processing}, pages 127--137, 2003.

\bibitem{gu2008computational}
Xianfeng~David Gu and Shing-Tung Yau.
\newblock {\em Computational conformal geometry}, volume~1.
\newblock International Press, Somerville, MA, 2008.

\bibitem{eirene}
G.~{Henselman} and R.~{Ghrist}.
\newblock Matroid filtrations and computational persistent homology.
\newblock {\em ArXiv e-prints}, June 2016.

\bibitem{Kapur2020}
Max Kapur.
\newblock Param tools, 2020.

\bibitem{lam2014landmark}
Ka~Chun Lam and Lok~Ming Lui.
\newblock Landmark-and intensity-based registration with large deformations via
  quasi-conformal maps.
\newblock {\em SIAM Journal on Imaging Sciences}, 7(4):2364--2392, 2014.

\bibitem{lhuilier1812maemoire}
AJ~Lhuilier.
\newblock Maemoire sur la polyaedromaetrie, contenant une daemonstration
  directe du thaeoraeme d'euler sur les polyaedres, et un examen de diverses
  exceptions auxquelles ce thaeoraeme est assujetti (extrait par m. gergonne).
\newblock {\em Annales de mathaematiques pures et appliquaees par Gergonne
  III}, 1812.

\bibitem{mcinnes2018umap}
Leland McInnes, John Healy, and James Melville.
\newblock {UMAP}: Uniform manifold approximation and projection for dimension
  reduction.
\newblock {\em arXiv preprint arXiv:1802.03426}, 2018.

\bibitem{moor2020topological}
Michael Moor, Max Horn, Bastian Rieck, and Karsten Borgwardt.
\newblock Topological autoencoders.
\newblock In {\em International conference on machine learning}, pages
  7045--7054. PMLR, 2020.

\bibitem{nene1996columbia}
Sameer~A Nene, Shree~K Nayar, Hiroshi Murase, et~al.
\newblock Columbia object image library (coil-100).
\newblock {\em Report}, 1996.

\bibitem{niyogi2008finding}
Partha Niyogi, Stephen Smale, and Shmuel Weinberger.
\newblock Finding the homology of submanifolds with high confidence from random
  samples.
\newblock {\em Discrete \& Computational Geometry}, 39:419--441, 2008.

\bibitem{place_cells}
John O'Keefe.
\newblock Place units in the hippocampus of the freely moving rat.
\newblock {\em Experimental Neurology}, 51(1):78--109, 1976.

\bibitem{perea_projective}
Jose~A. Perea.
\newblock Multiscale projective coordinates via persistent cohomology of sparse
  filtrations.
\newblock {\em Discrete Comput. Geom.}, 59(1):175--225, 2018.

\bibitem{perea_circular}
Jose~A. Perea.
\newblock Sparse circular coordinates via principal {$\mathbb{Z}$}-bundles.
\newblock In Nils~A. Baas, Gunnar~E. Carlsson, Gereon Quick, Markus Szymik, and
  Marius Thaule, editors, {\em Topological Data Analysis}, pages 435--458,
  Cham, 2020. Springer International Publishing.

\bibitem{pomerleau2012challenging}
Fran{\c{c}}ois Pomerleau, Ming Liu, Francis Colas, and Roland Siegwart.
\newblock Challenging data sets for point cloud registration algorithms.
\newblock {\em The International Journal of Robotics Research},
  31(14):1705--1711, 2012.

\bibitem{rey2019disentanglement}
Luis Armando~P{\'e}rez Rey.
\newblock Disentanglement with hyperspherical latent spaces using diffusion
  variational autoencoders.
\newblock {\em Proceedings of Machine Learning Research}, 1:1--4, 2019.

\bibitem{roweis2000nonlinear}
Sam~T Roweis and Lawrence~K Saul.
\newblock Nonlinear dimensionality reduction by locally linear embedding.
\newblock {\em science}, 290(5500):2323--2326, 2000.

\bibitem{decoding_neural}
Erik Rybakken, Nils Baas, and Benjamin Dunn.
\newblock Decoding of neural data using cohomological feature extraction.
\newblock {\em Neural Computation}, 31(1):68--93, 01 2019.

\bibitem{schonsheck2019chart}
Stefan~C. Schonsheck, Jie Chen, and Rongjie Lai.
\newblock Chart auto-encoders for manifold structured data.
\newblock {\em arXiv preprint arXiv:1912.10094}, 2019.

\bibitem{schonsheck2022semi}
Stefan~C. Schonsheck, Scott Mahan, Timo Klock, Timo Cloninger, and Rongjie Lai.
\newblock Semi-supervised manifold learning with complexity decoupled chart
  autoencoders.
\newblock {\em arXiv preprint arXiv:2208.10570}, 2022.

\bibitem{stanley1975fibonacci}
Richard~P Stanley.
\newblock The {F}ibonacci lattice.
\newblock {\em Fibonacci Quart}, 13(3):215--232, 1975.

\bibitem{ctralie2018ripser}
Christopher Tralie, Nathaniel Saul, and Rann Bar-On.
\newblock {Ripser.py}: A lean persistent homology library for python.
\newblock {\em The Journal of Open Source Software}, 3(29):925, Sep 2018.

\bibitem{tsai2007dimensionality}
Flora~S Tsai and Kap~Luk Chan.
\newblock Dimensionality reduction techniques for data exploration.
\newblock In {\em 2007 6th International conference on information,
  communications \& signal processing}, pages 1--5. IEEE, 2007.

\bibitem{van2008visualizing}
Laurens Van~der Maaten and Geoffrey Hinton.
\newblock Visualizing data using t-sne.
\newblock {\em Journal of machine learning research}, 9(11), 2008.

\bibitem{weisstein2008spherical}
Eric~W Weisstein.
\newblock Spherical trigonometry.
\newblock {\em From MathWorld--A Wolfram Web Ressource. Available from
  http://mathworld. wolfram. com/SphericalTrigonometry. html}, 2008.

\bibitem{wiki:Young_Fibonacci_lattice}
Wikipedia.
\newblock {Young–Fibonacci lattice} --- {W}ikipedia{,} the free encyclopedia.
\newblock
  \url{http://en.wikipedia.org/w/index.php?title=Young\%E2\%80\%93Fibonacci\%20lattice&oldid=1122161521},
  2023.
\newblock [Online; accessed 29-March-2023].

\bibitem{analogous_bars}
Hee~Rhang Yoon, Robert Ghrist, and Chad Giusti.
\newblock Persistent extension and analogous bars: Data-induced relations
  between persistence barcodes.
\newblock 2022.

\bibitem{zeng2010ricci}
Wei Zeng, Dimitris Samaras, and David Gu.
\newblock Ricci flow for 3{D} shape analysis.
\newblock {\em IEEE Transactions on Pattern Analysis and Machine Intelligence},
  32(4):662--677, 2010.

\bibitem{Zomorodian_Carlsson}
Afra Zomorodian and Gunnar Carlsson.
\newblock Computing persistent homology.
\newblock {\em Discrete Comput. Geom.}, 33(2):249--274, 2005.

\end{thebibliography}


\appendix

%\section{Additional Implementation Details}

\section{Algorithm} \label{apdx:alg}
\begin{algorithm}[H]
\caption{Energy Minimization }\label{alg:cap}
\begin{algorithmic}
\Require $\sk_2(X)$, $B$, $\alpha_0$, $r$
\State \textbf{Initialize} $p = [1,0,0]^N$  \Comment{Begin with all points on basepoint}
\State i = 0
\While{Not Converged}
    \State $\mathcal{E} \gets 0$ and $\nabla_{p \perp} \mathcal{E} \gets 0$
    
    \For {$x \in \sk_2(X)^2$}  \Comment{loop through triangles}
        \State $\mathcal{E} \gets \mathcal{E} + \mathrm{E}(x)$ 
        \For {$p \in x$}
            \State $g \gets p-B(x)$
            \State $g \gets g - \langle g, p \rangle$
            \State $g \gets \dfrac{g}{||g||_2}$
            \State $\nabla_{p \perp} \mathcal{E} (p) = \mathrm{E}(x) g$ \Comment{Update gradient}
        \EndFor
    \EndFor
    \State $p \gets p - r (\nabla_{p \perp} \mathcal{E})$
    \State $p \gets \dfrac{p}{||p||_2}$
    \State $B \gets \dfrac{\sum_{p \in x} p}{||\sum_{p \in x} p||_2}$
 \Comment{Find barycenters of shifted triangles}
    
    \If {$i \geq 50$}  \Comment{Begin mass centering after points have spread from the initial position}
        \State $c \gets \sum_p (p)$
        \State $p \gets p - rc$
        \State $p \gets \dfrac{p}{||p||_2}$
        \State $B \gets B - rc$  \Comment{Barycenters shift the same as points}
        \State $B \gets \dfrac{B}{||B||_2}$
    \EndIf
    
    \State i++
    
\EndWhile
\end{algorithmic}
\end{algorithm}

\end{document}